\documentclass[11pt]{amsart}
\usepackage{graphicx} 

\usepackage[english]{babel}
\usepackage{import}
\usepackage{amsmath,amsfonts,amsthm,amssymb,stmaryrd}

\usepackage[alphabetic,abbrev]{amsrefs} 
\renewcommand{\MR}[1]{}

\newcommand\jxbar{%
\draw (\thejxx*\w+0.5*\w,-\thejxy*\h) -- (\thejxx*\w+0.5*\w,-\thejxy*\h-\h);%
\stepcounter{jxx};%
}

\newcommand\jxcap{%
\draw (\thejxx*\w+0.5*\w,-\thejxy*\h-\h) to[out=90,in=90] (\thejxx*\w+1.5*\w,-\thejxy*\h-\h);%
\addtocounter{jxx}{2};%
}
\newcommand\jxcup{%
\draw (\thejxx*\w+0.5*\w,-\thejxy*\h) to[out=-90,in=-90] (\thejxx*\w+1.5*\w,-\thejxy*\h);%
\addtocounter{jxx}{2};%
}

\newcommand\jxnewline{\stepcounter{jxy}; \setcounter{jxx}{1}; }
\newcommand\jxspace[1]{\addtocounter{jxx}{#1}}

\newcommand\jxbox[2]{%
\node[anchor=north west,minimum height={\h}, minimum width={#1*\w},align=center,inner sep=0pt,draw=black,font=\small] at (\thejxx*\w,-\thejxy*\h) {#2};%
\addtocounter{jxx}{#1};
}

\newcounter{jxx}
\newcounter{jxy}

\newcommand\stringdiag[1]{
~\tikz[remember picture, looseness=2.5, baseline={([yshift=-0.5ex]current bounding box.center)}]\bgroup
\def\w{10pt}
\def\h{14pt}
\setcounter{jxx}{1}
\setcounter{jxy}{1}
#1 
\egroup~
}

\usepackage{url}
\usepackage{fancyhdr}
\usepackage{graphicx}
\usepackage{microtype}
\usepackage{wrapfig, caption}
\usepackage[hidelinks]{hyperref}
\usepackage{xcolor}
\usepackage{geometry}
\usepackage{tikz} \usetikzlibrary{cd,calc}
\usetikzlibrary{arrows}
\usetikzlibrary{decorations.markings}
\usepackage{dutchcal}

\usepackage{enumitem}
\setlist[itemize]{leftmargin=*,labelindent=5pt}

\usepackage{zref-clever} \newcommand\Cref[1]{\zcref[S]{#1}} 

\usetikzlibrary{decorations.markings}
\makeatletter
\tikzset{nomorepostaction/.code={\let\tikz@postactions\pgfutil@empty}}
\makeatother
\tikzset{
	partition/.style={
		scale=0.4,
		yscale=-1,
		baseline={([yshift=-0.5ex]current bounding box.center)}
	}
}

\newcommand\newtheoremx[2]{%
\AddToHook{env/#1/begin}{
\zcsetup{countertype={theorem=#1}}}
\newtheorem{#1}[theorem]{#2}
}

\newtheorem{theorem}{Theorem}[section]
\newtheoremx{claim}{Claim}
\newtheoremx{proposition}{Proposition}
\newtheoremx{corollary}{Corollary}
\newtheoremx{lemma}{Lemma}
\newtheoremx{conjecture}{Conjecture}

\newtheorem{introtheorem}{Theorem}

\zcsetup{countertype={introtheorem=theorem}}

\theoremstyle{definition}
\newtheoremx{definition}{Definition}
\newtheoremx{notation}{Notation}
\newtheoremx{assumption}{Assumption}

\theoremstyle{remark}
\newtheoremx{example}{Example}
\newtheoremx{remark}{Remark}
\newtheoremx{question}{Question}

\title[Abelian envelopes for interpolation categories of wreath products]{Abelian envelopes for interpolation categories of wreath products from monoidal adjunctions }

\author{Johannes Flake}
\address{Mathematical Institute, University of Bonn, Endenicher Allee 60, 53115 Bonn, Germany}
\email{flake@math.uni-bonn.de}

\author{Thorsten Heidersdorf}
\address{School of Mathematics, Statistics and Physics, Newcastle University, Newcastle, United Kingdom}
\email{thorsten.heidersdorf@newcastle.ac.uk}

\author{David Hull}
\address{Mathematical Institute, University of Bonn, Endenicher Allee 60, 53115 Bonn, Germany}
\email{hulld@tcd.ie}

\date{}

\newcommand\cA{\mathcal A}
\newcommand\cC{\mathcal C}
\newcommand\cD{\mathcal D}
\newcommand\cE{\mathcal E}

\newcommand\cT{\mathcal T}
\newcommand\cS{\mathcal S}
\newcommand\kk{\Bbbk}
\newcommand\Rep{\operatorname{Rep}}
\newcommand\Res{\operatorname{Res}}
\newcommand\uRep{\underline{\operatorname{Rep}}}
\newcommand\Id{\operatorname{Id}}
\newcommand\id{\operatorname{id}}
\newcommand\im{\operatorname{im}}
\renewcommand\:\colon
\renewcommand\o\otimes
\newcommand\one{\mathbf{1}}
\newcommand\low{\operatorname{low}}
\newcommand\ol\overline

\newcommand\Inv{\operatorname{Inv}}
\newcommand\cG{\mathcal G}
\newcommand\ZZ{\mathbb Z}

\renewcommand\o\otimes
\newcommand\cO{\mathcal O}
\newcommand\ev{\operatorname{ev}}
\newcommand\coev{\operatorname{coev}}
\newcommand\tforall{\text{for all }}
\renewcommand\:\colon
\newcommand\Ps{\operatorname{Ps}}
\newcommand\Rel{\operatorname{Rel}}
\newcommand\gr{\operatorname{gr}}

\newcommand\GSet{G\text{-}\mathrm{Set}}
\newcommand\op{^{\mathrm{op}}}

\usepackage{tikz}

\tikzset{
    partition/.style={
      scale=0.45,
      yscale=-1,
      baseline={((0,-0.35))}
    }
}

\tikzset{
    bend/.cd,
    0/.style={},
    1/.style={bend right},
    -1/.style={bend left}
}
\newcommand\makePartPt[1]{({Mod(#1,10)},{(#1-Mod(#1,10))*.1})}
\newcommand\makePartLn[2]{%
\pgfmathtruncatemacro\bend{%
(int(#1/10)==int(#2/10)) ?
(#1<10 ? 1 : -1)*(#1>#2 ? 1 : -1)
: 0%
}
\draw[draw=black,line width=0.5pt,line cap=round] ({mod(#1,10)},{int(#1/10}) to[bend/\bend] ({mod(#2,10)},{int(#2/10)});
}
\def\cmddots{dots}
\newcommand\tpx[3] {%
~\tikz[partition] {
\draw[white,opacity=0] (1,0)--(1,1); 
\def\j{0}
\foreach \i [remember=\i as \j] in {#1} {
  \ifnum \i>0
    \ifnum \j>0
      \makePartLn{\i}{\j};
    \fi
  \fi
} 
\foreach \i [remember=\i as \j] in {#1} {
  \ifnum \i>0
    \draw[draw=none,fill=black] \makePartPt\i circle (3pt);
  \fi
} 
\foreach \t [count=\x] in {#2}
{
  \ifx\t\cmddots \node[above=-2pt,font=\scriptsize,fill=white,inner xsep=0pt,inner ysep=5pt] at \makePartPt{\x} {...};
  \else \node[above=2pt,font=\scriptsize] at \makePartPt{\x} {\t};
  \fi
} 
\foreach \t [count=\x] in {#3}
{
  \ifx\t\cmddots \node[below=-2pt,font=\scriptsize,fill=white,inner xsep=0pt,inner ysep=5pt] at \makePartPt{10+\x} {...};
  \else \node[below=2pt,font=\scriptsize] at \makePartPt{10+\x} {\t};
  \fi
 } 
}~}

\newcommand\tp[1]{\tpx{#1}{}{}}

\newcommand\listorempty[1]{\def\temp{#1}\ifx\temp\empty \emptyset \else (#1) \fi}

\begin{document}

\begin{abstract} We establish the existence of abelian envelopes for interpolation categories of wreath product groups $G\wr S_n$, for a fixed finite group $G$ with the symmetric groups $S_n$, for $n\ge0$. Our approach consists of showing directly via essentially combinatorial methods that certain generalized restriction functors admit adjoints. 
\end{abstract}

\maketitle


\section{Introduction}

In modern representation theory, it has been realized that it can be very helpful to study the collection of representations of an algebraic object in its entirety, that is, to describe categories of representations including the natural structures they carry, such as tensor products. In this spirit, studying the class of tensor categories, which share sufficiently many structural features with, say, categories of representations of finite or algebraic groups, can be considered an extension of classical representation theory. It is in this realm that the family $(\Rep S_n)_{n\ge0}$ of categories of representations of the symmetric groups can be interpolated by categories $\uRep(S_t)$ due to Deligne \cite{Deligne}, where $t$ is an interpolation parameter, for instance, from the complex numbers. 

While such interpolation categories typically admit convenient descriptions based on some kind of diagrammatic calculus, they are generally not abelian categories, and the existence of a``minimal'' abelian tensor category containing a given interpolation category, the \emph{abelian envelope}, is often a challenging problem. For the Deligne's interpolation categories of the symmetric groups, this was established in \cite{CO-ab}.

While there are general existence criteria for the existence of abelian envelopes  \cites{Cou-ab,CEOP,CEAH}, they are often difficult to verify in practice. The goal of this paper is to establish the existence of abelian envelopes for interpolation categories $\cC_t$ for the family of wreath product groups $(G\wr S_n)_{n\ge0}$, for a fixed finite group $G$ and interpolation parameter $t\in\kk$, where $\kk$ is an algebraically closed field of characteristic $0$. The categories $\cC_t$ arise as special cases of Knop's tensor envelopes \cite{Knop}, and also as special cases of a general class of interpolation categories of wreath products due to Mori \cite{Mori}. Incarnations of the categories $\cC_t$ have been studied from various points of view in \cites{Repcomplexrank,Freslon-Skalski,HT, Harman1,Ryba1,Ryba2}. We follow the diagrammatic definition of the categories $\cC_t$ from \cite{LS}, which allows us to think of morphisms in $\cC_t$ as linear combinations of partition diagrams whose vertices are labeled by elements from $G$ (see \Cref{sec:G-partitions}). 

To show the existence of abelian envelopes, we utilize a recently suggested technique involving monoidal adjunctions \cite{AbEnv1}. More precisely, we consider generalized ``restriction'' functors for the interpolation categories $\cC_t$, in the following sense:

\begin{introtheorem}[\Cref{thm:res}] \label{thm:A} For all $n\in\ZZ_{\ge0}$ and $t\in\kk$, there are linear symmetric monoidal functors of the form
$$
\cC_t \to \cC_{t-n} \boxtimes \Rep(G\wr S_n)
$$
that have left adjoint functors and right adjoint functors.     
\end{introtheorem}

The functors appearing in \Cref{thm:A} were constructed in \cite{Mori} in a more general situation. Our construction of the functors is very explicit and diagrammatic. Technically, it relies on a general idea of linear symmetric monoidal functors constructed from \emph{mappings of invariants}, which we develop in \Cref{sec:lsm-functors}. The existence of the adjoint (``induction'') functors in \Cref{thm:A}, which are only op-/lax monoidal functors, is then established using a general theory of adjoints of linear monoidal functors, as developed in \cite{AbEnv1}, and some explicit combinatorial computations involving vertex-labeled partitions. Even though interpolation versions of induction functors are discussed in \cite{Mori}, it does not seem to be considered if they are actually adjoints of the relevant generalized restriction functors (see \Cref{rem:comparison-Mori}).

The categories $\cC_t$ generalize Deligne's interpolation categories $\uRep(S_t)$ for the symmetric groups, which are obtained by specializing $G$ to the trivial group. In this special case, the above restriction and induction functors are discussed in \cites{Deligne,Repcomplexrank,HK}. For $G$ a finite cyclic group and $n=1$, \Cref{thm:A} was one of the main results in the Master thesis of one of the authors \cite{Thesis}. Our proof strategy of \Cref{thm:A}, as well as the proof in \cite{Thesis}, is based on a proof of \Cref{thm:A} (using techniques from \cite{AbEnv1}) for $G=\{1\}$ that is contained in \cite{AbEnv2} (in preparation), an early version of which was available to the authors.

Using recent general ideas on abelian envelopes from \cites{AbEnv1,FG}, we obtain from \Cref{thm:A} the following result in our situation:

\begin{introtheorem}[\Cref{thm:abenv}] For any choice of $G$ and any $t\in\kk$, where $\kk$ is algebraically closed of characteristic $0$, $\cC_t$ has an abelian envelope with enough projectives. The abelian envelope is a lower finite highest weight category in the sense of Brundan--Stroppel \cite{BS} whose category of tilting modules can be identified with $\cC_t$. 
\end{introtheorem}

An alternative route towards abelian envelopes for the categories $\cC_t$ would be via a recently developed theory of tensor categories from pro-oligomorphic groups \cite{HS}, as explained in \Cref{rem:HS}.

We conclude this paper with an appendix explaining how the categories $\cC_t$ can be realized using a diagrammatic calculus involving edge labels instead of vertex labels, and how they are obtained via Knop's construction in \cite{Knop}.

\bigskip\par\noindent\textbf{Acknowledgments.} We thank Robert Laugwitz and Sebastian Posur for agreeing to share an early version of the joint paper with J.F. \cite{AbEnv2} (in preparation) with the authors of this paper. J.F.~thanks the Max Planck Institute for Mathematics Bonn for the excellent conditions
provided while some early ideas relevant for this paper were discussed with T.H.


\section{LSM functors} \label{sec:lsm-functors}

Throughout, $\kk$ is a field of characteristic $0$.
We will abbreviate the predicates ``$\kk$-linear symmetric monoidal'' to \emph{LSM}. An LSM category is called \emph{pseudo-abelian} if it is additive and idempotent complete.
For any LSM category $\cC$, the \emph{pseudo-abelian completion}, i.e., the additive and idempotent completion, is denoted $\Ps(\cC)$. A \emph{pseudo-tensor category}, for us, is a hom-finite pseudo-abelian rigid LSM category whose tensor unit has endomorphism ring $\kk$, while a \emph{tensor category}, for us, is a pseudo-tensor category that is abelian and whose objects are all of finite length. In particular, here, all (pseudo-)tensor categories are symmetric.

We begin by establishing a technique that will allow us the construction of LSM functors and equivalences.

\begin{definition} \label{def:invariants} For any LSM category $\cC$ and any object $V\in\cC$, we define
$$
\Inv_d(V):=\cC(\one,V^{\o d}),\text{ for all }d\ge0,\quad
\Inv(V):=\bigoplus_{d\ge0} \Inv_d(V).
$$
\end{definition}

We can view $\Inv(V)$ as a $\ZZ_{\ge0}$-graded $\kk$-algebra, with multiplication given by the tensor product in $\cC$ and such that the degree-$d$ subspace $\Inv_d(V)$ is an $S_d$-module with an action given by the symmetric braiding in $\cC$.

Note that if $\cC$ is the category of representations of a group and $V$ is an object, then $\Inv(\cC)$ can be identified with the space of tensor invariants, i.e., invariants under the group action in the tensor algebra of $V$.

For any LSM category $\cC$ and any object $V\in\cC$ we denote by $\langle V\rangle_\o$ the full subcategory on tensor powers of $V$, which we view as an LSM category, as well.

\begin{definition} \label{def:mapping-invariants} A \emph{mapping of invariants} is a tuple $(V,\ev_V,\coev_V,W,\ev_W,\coev_W,F_\one)$ where $(V,\ev_V,\coev_V)$ and $(W,\ev_W,\coev_W)$ are self-dual objects in LSM categories $\cC$ and $\cD$, respectively, with (co-)evaluation morphisms as indicated, and $F_\one\:\Inv(V)\to\Inv(W)$ is a graded algebra map equivariant with respect to the symmetric group actions (in each degree) satisfying 
$$
F_\one(\coev_V)=\coev_W,
$$
$$
F_\one((\ev_V\o V^{\o(d-2)}) u)
= 
(\ev_W\o W^{\o(d-2)}) F_\one(u) 
\quad\text{for all }d\ge2, u\in\Inv_d(V).
$$  
We also denote the above tuples of data as $F_\one\:\Inv(V,\ev_V,\coev_V)\to\Inv(W,\ev_W,\coev_W)$.
\end{definition}

\begin{proposition} \label{prop:functors} Let $F_\one\:\Inv(V,\ev_V,\coev_V)\to\Inv(W,\ev_W,\coev_W)$ be a mapping of invariants, in the sense of \Cref{def:mapping-invariants}. Then there is a unique LSM functor $F\:\langle V\rangle_\o\to\cD$ such that $F(V)=W$, $F(\ev_V)=\ev_W$, and $F(u)=F_\one(u)$, for all $u\in\Inv(\cC)$. Moreover, $F$ is faithful / full if and only if $F_\one$ is injective / surjective, respectively.
\end{proposition}

\begin{proof} For any $n\ge0$, we define the morphisms in $\cC$
$$
\ev^{(n)}_V := \ev_V (V\o \ev_V\o V) \dots (V^{\o(n-1)}\o \ev_V\o V^{\o(n-1)}) ,
$$
$$
\coev^{(n)}_V := (V^{\o(n-1)}\o \coev_V \o V^{\o(n-1)}) \dots (V\o \coev_V\o V) \coev_V ,
$$
which can be taken as co-/evaluations of the self-dual objects $V^{\o n}$. Similarly, we define morphisms $\ev_W^{(n)}$ and $\coev_W^{(n)}$ in $\cD$. 
Note that, if a functor $F$ as asserted exists, then in particular, it has to send $\coev^{(n)}_V$ to $\coev^{(n)}_W$, for all $n$.

Next we observe, that if $F$ as asserted exists, then it is compatible with post-composing evaluations even in the sense that
$$
F_\one( (V^{\o i} \o \ev_V \o V^{\o(d-2-i)}) u)
 = (W^{\o i} \o \ev_W \o W^{\o(d-2-i)})F_\one(u)
$$
for all $d\ge 2$, $0\le i\le d-2$, $u\in\Inv_d(V)$. This follows from the case $i=0$, as assumed, using that $F$ is symmetric and that $V\o \ev_V$ is related to $\ev_V\o V$ by a braiding in $\cC$.

For any $m,n\ge0$, consider $f\in\cC(V^{\o m},V^{\o n})$. We can write
$$
f = (V^{\o n} \o \ev^{(m)}_V) (f\o V^{\o m}\o V^{\o m}) (\coev^{(m)}_V\o V^{\o m}) .
$$
Hence, if a functor $F$ as asserted exists, then it has to send $f$ to
$$
y_f := (W^{\o n}\o \ev^{(m)}_W) \Big(F_\one( (f\o V^{\o m})\coev^{(m)}_V ) \o W^{\o m} \Big) .
$$
We can now verify that the assignment $f\mapsto y_f$ indeed defines a functor. First, if $f$ is the identity on the object $V^{\o n}$, then
$$
y_f = (W^{\o n}\o \ev^{(n)}_W) (\coev^{(n)}_W\o W^{\o n})
$$
is the identity on $W^{\o n}$. Also, for any $m,n,k\ge0$ and morphisms $V^{\o m}\xrightarrow{f} V^{\o n} \xrightarrow{g} V^{\o k}$, we compute
\begin{align*}
y_g y_f 
=& (W^{\o k}\o \ev^{(n)}_W) \Big(F_\one( (g\o V^{\o n})\coev^{(n)}_V ) \o W^{\o n} \Big) 
\\
&\qquad (W^{\o n}\o \ev^{(m)}_W) \Big(F_\one( (f\o V^{\o m})\coev^{(m)}_V ) \o W^{\o m} \Big) 
\\
=& (W^{\o k}\o \ev^{(n)}_W\o \ev^{(m)}_W) 
\Big(F_\one( (g\o V^{\o n}) \coev^{(n)}_V ) \o F_\one ((f\o V^{\o m}) \coev^{(m)}_V) \o W^{\o m} \Big)
\\
=& (W^{\o k}\o \ev^{(n)}_W\o \ev^{(m)}_W) 
\Big(F_\one( ((g\o V^{\o n}) \coev^{(n)}_V) \o ((f\o V^{\o m}) \coev^{(m)}_V)) \o W^{\o m} \Big)
\\
=& (W^{\o k} \o \ev^{(m)}_W) ( F_\one(
(V^{\o k}\o \ev^{(n)}_V\o V^{\o m})(g\o V^{\o n}\o f\o V^{\o m}) (\coev^{(n)}_V \o \coev^{(m)}_V)) \o W^{\o m})
\\
=& (W^{\o k} \o \ev^{(m)}_W) ( F_\one(
((gf)\o V^{\o m}) \coev^{(m)}_V) \o W^{\o m})
\\
=& y_{gf} ,
\end{align*}
see also \Cref{fig1}. We have expanded the definitions in step 1, re-written some morphisms in $\cC$ in step 2, used that $F_\one$ is an algebra morphism in step 3, used that $F_\one$ is compatible with post-composing evaluations in step 4, and used a snake relation in $\cD$ in step 5. We have shown that the functor does exist and is unique.

\begin{figure}[ht]
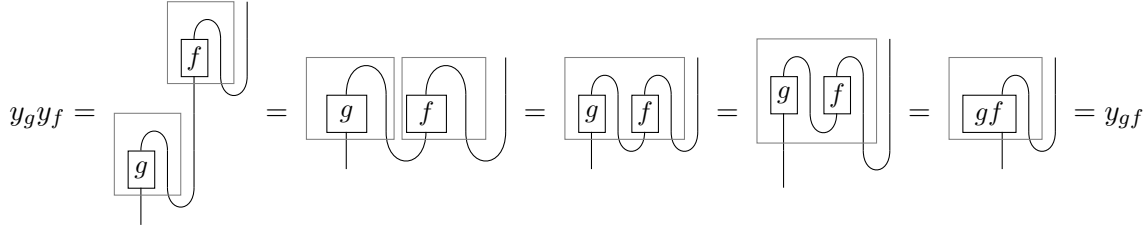

$$
y_g y_f = \stringdiag{ 
\jxspace{2} \jxcap \jxbar \jxnewline
\jxspace{2} \jxbox{1}{$f$} \jxbar \jxbar \jxnewline
\jxspace{2} \jxbar \jxcup \jxnewline
\jxcap \jxbar \jxnewline
\jxbox{1}{$g$} \jxbar \jxbar \jxnewline
\jxbar \jxcup
\draw[black!50] (\w*2.5,-\h) rectangle (\w*5,-\h*3.2);
\draw[black!50] (\w*0.5,-\h*4) rectangle (\w*3,-\h*6.2);
} = \stringdiag{ \def\w{15pt}
\jxcap \jxcap \jxbar \jxnewline
\jxbox{1}{$g$} \jxbar \jxbox{1}{$f$} \jxbar \jxbar \jxnewline
\jxbar \jxcup \jxcup
\draw[black!50] (\w*0.5,-\h) rectangle (\w*2.7,-\h*3.2);
\draw[black!50] (\w*2.9,-\h) rectangle (\w*5,-\h*3.2);
} = \stringdiag{ 
\jxcap \jxcap \jxbar \jxnewline
\jxbox{1}{$g$} \jxbar \jxbox{1}{$f$} \jxbar \jxbar \jxnewline
\jxbar \jxcup \jxcup
\draw[black!50] (\w*0.5,-\h) rectangle (\w*5,-\h*3.2);
} = \stringdiag{ 
\jxcap \jxcap \jxbar \jxnewline
\jxbox{1}{$g$} \jxbar \jxbox{1}{$f$} \jxbar \jxbar \jxnewline
\jxbar \jxcup \jxbar \jxbar \jxnewline
\jxbar \jxspace{2} \jxcup
\draw[black!50] (\w*0.5,-\h) rectangle (\w*5,-\h*3.8);
} = \stringdiag{ 
\jxspace{1} \jxcap \jxbar \jxnewline
\jxbox{2}{$gf$} \jxbar  \jxbar \jxnewline
\jxspace{1} \jxbar \jxcup
\draw[black!50] (\w*0.5,-\h) rectangle (\w*4,-\h*3.2);
}
= y_{gf}
$$    
\caption{Visualization of a computation using the graphical calculus of a rigid category. Here, composition is read top-to-bottom. The gray box represents the mapping $F_\one$.}
\label{fig1}
\end{figure}

The assertions on it being faithful or full follow using the isomorphisms
$$
\cC(V^{\o m},V^{\o n}) \xrightarrow{\cong} \cC(V^{\o(m+n)},\one), \quad
f\mapsto (f\o V^{\o m})\coev^{(m)}_V , \quad 
(V^{\o n}\o \ev^{(m)}_V)(g\o V^{\o n}) \mapsfrom g . 
$$
\end{proof}

To see that a map $F_\one$ as above is surjective in applications, we use the following reduction criterion. 

\begin{definition} \label{def:reduced-invariants}
In the set-up as before, define
$$
\Inv_{<d}(V):=\kk S_d\,\sum_{d_1,d_2\ge1, d_1+d_2=d} \Inv_{d_1}(V)\o\Inv_{d_2}(V)\subset\Inv_d(V)
$$
$$
\ol\Inv_d(V) := \frac{\Inv_d(V)}{\Inv_{<d}(V)} \quad
\text{for all $d\ge0$}.
$$
\end{definition}

\begin{lemma} \label{lem:reduction-surjective} A mapping of invariants $F_\one\:\Inv(V,\ev_V,\coev_V)\to\Inv(W,\ev_W,\coev_W)$ is surjective if the induced map $\Inv_d(V)\xrightarrow{F_\one}\ol\Inv_d(W)$ is surjective for all $d\ge1$.    
\end{lemma}

\begin{proof} The assumptions mean that, for any $f\in\Inv_d(W)$, there is a $g\in\Inv_d(V)$ such that $f-F_\one(g)\in\Inv_{<d}(V)$, but $\Inv_{<d}(V)$ is contained in the image of $F_\one$ by induction in $d$.
\end{proof}

\section{Categories from group partitions} 
\label{sec:interpolation}

\subsection{Background: Categories from partitions} \label{bg:partitions} In \cite{Deligne}, see also \cite{CO-blocks}*{Section~2}, a family of LSM categories based on partition diagrams is defined. We briefly review their construction.

For any $a,b\in\ZZ_{\ge0}$, a \emph{partition diagram} is an equivalence class of string diagrams with $a$ upper and $b$ lower points (vertices) and arcs (edges) connecting pairs of these points. The equivalence relation identifies two such diagrams if and only if they correspond to the same set partition of the set $\{1,\dots,a\}\sqcup\{1,\dots,b\}$ of upper and lower points. Let $P_{a,b}$ denote the set of partition diagrams with $a$ upper and $b$ lower points.

For any $t\in\kk$, the $\kk$-linear symmetric strict monoidal category $\cS^0_t$ has objects denoted $[a]$, for $a\in\ZZ_{\ge0}$, and hom-spaces $\cS_t^0([a],[b]):=\kk P_{a,b}$. Its identities are given by the partition diagrams in $P_{a,a}$ obtained by connecting the $i$-th upper and the $i$-th lower point, for all $1\le i\le a$. The composition is defined on pairs $(g,f)$ of partition diagrams in $P_{b,c}\times P_{a,b}$ by placing $f$ one on top of $g$, gluing the diagrams along the $b$ common middle points, and removing every component of the resulting diagram that has no upper or lower points, replacing it by a factor $t$.

Its tensor unit is the object $[0]$. Tensor products are defined on pairs $(f,g)$ of partition diagrams as the map $P_{a,b}\times P_{a',b'}\to P_{a+a',b+b'}$ given by placing $f$ to the left of $g$. The symmetric braiding $c_{a,b}\:[a]\o[b]\to[b]\o[a]$ is defined by the partition diagram obtained by connecting the first $a$ upper points pairwise with the $a$ last lower points, and connecting the $b$ first lower points pairwise with the $b$ last upper points.

The pseudo-abelian completion $\cS_t:=\Ps(\cS^0_t)$ is called Deligne's interpolation category and is also denoted $\uRep(S_t)$ in the literature.
It is an interpolation category for the family of symmetric groups $(S_n)_{n\ge0}$ in the following sense: $\cS_t$ is a semisimple tensor category for all $t\not\in\ZZ_{\ge0}$, while it is a pseudo-tensor category whose unique minimal, or equivalently semisimple quotient category is $\Rep(S_n)$, the ordinary category of representations, for all $t=n\in\ZZ_{\ge0}$.

\subsection{Group partitions}\label{sec:G-partitions}

For the rest of the section, we fix a group $G$.
We recall the set-up of \cite{LS}*{Section~3}.

\begin{definition} A \emph{$G$-partition} or \emph{$G$-diagram} is an equivalence class of partition diagrams whose vertices are labeled by elements of $G$. Two diagrams are equivalent if and only if the underlying partition diagrams are equivalent, and for each component in this common partition diagram, there is a $g\in G$ such that the labels in the two diagrams are related by applying $g\cdot$.
\end{definition}

We will sometimes omit the trivial label $1\in G$. For all integers $a,b\ge0$, we denote the set of $G$-diagrams with $a$ upper and $b$ lower points by $P_{a,b}$. This agrees with the definition in \Cref{bg:partitions} for $G=\{1\}$.

\begin{example} \label{ex:labeled-diags} For any $a\ge0$, we set
$
\id_a := \underbrace{\tp{1,11}\dots\tp{1,11}}_{a} = \underbrace{\tpx{1,11}{1}{1}\dots\tpx{1,11}{1}{1}}_{a} 
\in P_{a,a}
$.

For any $d\ge0$, $k_1,\dots,k_d\in G$, we have a $G$-diagram
$
\tpx{11,12,13,14}{}{$k_1$,$k_2$,dots,$k_d$}
\in P_{0,d}
$.
\end{example}

We recall two operations on ($\kk$-linear combinations of) $G$-diagrams, a tensor product and a composition, the latter depending on a parameter in $\kk$. 

For $a,b,c\ge0$, $f\in P_{a,b}$, and $g\in P_{b,c}$, we say $f,g$ are \emph{compatible} if a representing labeled partition diagram can be chosen in each equivalence class such that the labels for the $b$ lower points of $f$ coincide with those for the upper $b$ points of $g$. If $f,g$ are compatible, then we can glue two such representatives together along the common $b$ upper / lower points, resulting in a labeled partition diagram possibly with components without upper or lower points. Let $\mu(f,g)\in\ZZ_{\ge0}$ denote the number of such components, and let $f*g\in P_{a,c}$ denote the $G$-diagram represented by the labeled partition diagram obtained by removing those $\mu(f,g)$ components.

Also, for any $a,b,a',b'\ge0$, $f\in P_{a,b}$, $g\in P_{a',b'}$, we let $f\o g\in P_{a+a',b+b'}$ be an $G$-diagram represented by the labeled partition diagram on $a+a'$ upper and $b+b'$ lower points which is given by any representative of $f$ on the leftmost upper and lower points, and by any representative of $g$ on the rightmost upper and lower points, for all $f\in P_{a,b}$, $g\in P_{a',b'}$.

It is shown in \cite{LS}*{Section~3}, that with these definitions, $\mu(f,g)$, $f*g$, and $f\o g$ do not depend on the choices of representatives.
Hence, for any $t\in\kk$ and integers $a,b,c\ge0$, there is a unique linear map
$$ \circ\: \kk P_{b,c}\o_\kk \kk P_{a,b}\to \kk P_{a,c}
$$
sending $g\o_\kk f$ to  $(mt)^{\mu(f,g)} (f*g)$ if $f,g$ are compatible, and to $0$ otherwise, for all $f\in P_{a,b}$, $g\in P_{b,c}$. We call this family of maps \emph{$t$-composition}. 
Similarly, for any integers $a,b,a',b'\ge0$, there is a unique linear map
$$ \otimes\: \kk P_{a,b}\o_\kk \kk P_{a',b'}\to \kk P_{a+a',b+b'}
$$
sending $f\o_\kk g$ to $f\o g$, as defined above, for all $f\in P_{a,b}$, $g\in P_{a',b'}$.

\begin{example}[compositions and tensor products]
For $G=C_5$, the cyclic group of order $5$ with elements $G=\{0,1,2,3,4\}$ (now $0$ is the identity element, not $1$), let $p\in P_{3,4}$ and $q\in P_{4,2}$ be the $C_5$-diagrams  
$$p=\tpx{01,02,0,11,12,03,0,13,14}{3,4,2}{0,1,4,2} ,\qquad
q=\tpx{01,0,02,11,0,03,04,0,12}{2,3,2,0}{2,1} ,
$$
$$\text{then}\qquad q\circ p =t\tpx{01,02,0,11,03,0,12}{3,4,2}{0,1}
\quad\text{and}\quad
p\otimes q=\tpx{01,02,0,11,12,03,0,13,14,0,04,0,05,15,0,15,0,06,07,0,16}{3,4,2,2,3,2,0}{0,1,4,2,2,1}.
$$
\end{example}

Set
$
\Psi_{a,b} := \tpx{1,14,0,3,16,0,4,11,0,6,13}{,$\dots$,,,$\dots$}{,$\underbrace{\dots}_{a}$,,,$\underbrace{\dots}_{b}$}
\in P_{a+b,a+b}
\quad\text{for all }a,b\ge0.
$

\begin{definition}[{\cite{LS}*{Definition~3.6}}] \label{def:categories} For any $t\in\kk$, let $\cG_t$ be the unique $\kk$-linear symmetric strict monoidal category whose objects are the set $\{[a]\}_{a\ge0}$ and whose morphism spaces are $\cG_t([a],[b])=\kk P_{a,b}$ with composition and tensor product as defined above, with identities given by the diagrams $\{\id_a\}_{a\ge0}$ as in \Cref{ex:labeled-diags}, with tensor unit $[0]$, and with a symmetric braiding given by $\Psi_{a,b}$. Let $\Ps(\cG_t)$ be the pseudo-abelian envelope of $\cG_t$.
\end{definition}

Note that on objects the tensor product is given by $[a]\o[b]=[a+b]$.

\subsection{Interpolation} In case $G$ is a finite group, it was shown in \cite{LS}*{Section~9}, that the categories from \Cref{def:categories} can be viewed as interpolation categories for the family of wreath product groups $(G\wr S_n)_{n\ge0}$. The argument there, however, is somewhat indirect. First, an equivalence of these categories with certain tensor envelopes in the sense of Knop \cite{Knop} is established, and then Knop's results imply the existence of the desired interpolation functors.

We will construct the interpolation functors directly, using the machinery of \Cref{sec:lsm-functors}, and will also explicitly describe their kernels, the ideals of negligible morphisms, as we will require such a concrete description in later parts of this paper. 

We now also fix $n\ge0$ and set $G_n:=G\wr S_n=G^n\rtimes S_n$, the wreath product. Set $m:=|G|$ and fix a bijection $\theta\:\{1,\dots,m\}\to G$. Then $G$ has a regular action on the set $G\simeq\{1,\dots,m\}$, using the fixed identification $\theta$. Hence, $G_n$ has a permutation action on the space of $n\times m$-matrices, where $S_n$ permutes the rows and the $i$-th copy of $G$ permutes the entries in the $i$-th row, for each $1\le i\le m$. Let $W$ be the $mn$-dimensional $G_n$-module defined by this action, so $W\in\Rep G_n$.

For any element $w'$ in some $G_n$-module $W'$, we denote its $G_n$-orbit sum by $\cO(w')$. For any $1\le i\le n$ and $1\le j\le m$, we denote by $E_{ij}$ the elementary matrix with entry $1$ at position $(i,j)$, viewed as an element in $W$. 

\begin{lemma} \label{lem:generators-invariants} For any $d\ge0$, $\ol\Inv_d(W)$ (as in \Cref{def:reduced-invariants}) is spanned by (the images of) the maps $$
1\mapsto
\cO(E_{1,k_1}\o\dots\o E_{1,k_d})
=\sum_{g\in G} g(E_{1,k_1}\o\dots\o E_{1,k_d}),
\quad\text{ for } 1\le k_1,\dots,k_d \le m .
$$
\end{lemma}

\begin{proof} For any $d\ge0$, the set 
$$
T_d := \{ E_{i_1,j_1}\o\dots\o E_{i_d,j_d}: 1\le i_1,\dots,i_d\le n, 1\le j_1,\dots,j_d\le m\}
$$
spans $W^{\o d}$, so the set $\{\cO(w):w\in T_d\}$ spans $\Inv_d(W)$. 

For any $w= E_{i_1,j_1}\o\dots\o E_{i_d,j_d}$ in $T_d$, we denote by $r(w):=\{i_1,\dots,i_d\}$ the set of row indices of elementary matrices involved in the tensor product, and by $\#r(w):=\#\{i_1,\dots,i_d\}$ the number of distinct rows. We claim that for any $d\ge0$, the set of $w\in T_d$ with $\#r(w)=1$ spans $\ol\Inv_d(W)$. Then the assertion follows, as any such $w$ has an element of the form $E_{1,k_1}\o\dots\o E_{1,k_d}$ in its $G_n$-orbit. 

To show the claim, fix $d$ and assume we know that the set of $w\in T_d$ with $\#r(w)\le n'$ spans $\ol\Inv_d(W)$ for some $n'\ge1$. This is the case for $n'=n$. We are done if $n'=1$, so let us assume $n'\ge2$. Consider any $w\in T_d$ with $\#r(w)=n'$. Replacing $w$ by an element in the same orbit, we may assume that $r(w)=\{1,\dots,n'\}$. Collecting all elementary matrices with the row index $1$ in the tensor product $w$, we can find a permutation $\sigma\in S_d$ and elements $w_1,w_2\in\bigsqcup_{d'} T_{d'}$ such that $r(w_1)=\{1\}$, $r(w_2)=\{2,\dots,n'\}$, and
$$
w = \sigma(w_1\o w_2) ,
$$
using the action of $S_d$ on $\Inv_d(W)$ by permuting tensor factors. We compute
\begin{align*}
\sigma(\cO(w_1)\o\cO(w_2))
&= \sigma\sum_{g_1,g_2\in G} g_1(w_1)\o g_2(w_2) 
= \sigma\sum_{g_1,g_2\in G: g_1(\{1\})\subset g_2(\{2,\dots,n'\})} g_1(w_1)\o g_2(w_2) \\
&\quad\quad\quad+\sigma\sum_{g_1,g_2\in G: g_1(\{1\})\not\subset g_2(\{2,\dots,n'\})} g_1(w_1)\o g_2(w_2)
\\
&\in \kk\{\cO(w'): w'\in T_d, \#r(w')\le n'-1\} + z\cO(w)
,
\end{align*}
for some $z\in\ZZ_{>0}$. Here, we use the action of $G_n$ on the set of rows (that factors via $S_n$) in the third expression, when splitting up the sum. As $\sigma(\cO(w_1)\o\cO(w_2))\in\Inv_{<d}(W)$, this shows that $\ol\Inv_d(W)$ is spanned by $\{\cO(w)\}$ for all $w\in T_d$ with $\#r(w)\le n'-1$, which implies the claim by induction.
\end{proof}

Recall that we have fixed an identification of ($G$-)sets $\theta\:\{1,\dots,m\}\to G$. 

\newcommand\supp{\operatorname{supp}}
\begin{definition} \label{def:supp}
For any $d\ge0$ and any $G$-partition $f\in P_{0,d}$ with (some choice of) labels $(k_1,\dots,k_d)$ from $G$, let $\supp(f)$ be the set of tuples $(i,j)$, where $i=(i_1,\dots,i_d)\in \{1,\dots,n\}^d$, $j=(j_1,\dots,j_d)\in\{1,\dots,m\}^d$ such that $i_p=i_q$ and $\theta(j_p)^{-1}\theta(j_q)=k_p^{-1} k_q$, whenever the lower points $p,q$ are in the same component of $f$.
\end{definition}

Note that this definition does not depend on the choice of labels within a given equivalence class.

Recall the definitions from \Cref{ex:labeled-diags}. We fix a choice of evaluation and coevaluation maps for the self-dual objects $V:=[1]\in\cG_t$ and $W\in\Rep G_n$ as follows:
$$
\ev_V = \tp{1,2},\quad \coev_V = \tp{11,12},\quad
\ev_W = (E_{i_1,j_1}\o E_{i_2,j_2}\mapsto \delta_{i_1,i_2}\delta_{j_1,j_2}),$$
$$
\coev_W = (1\mapsto \sum_{(i_1,j_1),(i_2,j_2)\in\{1,\dots,n\}\times\{1,\dots,m\}} E_{i_1,j_1}\o E_{i_2,j_2}).
$$

\begin{lemma} \label{lem:helper} If $t=n$, then the assignments
$$
P_{0,d} \ni f \mapsto \sum_{(i,j)\in\supp(f)} E_{i_1,j_1}\o\dots\o E_{i_d,j_d} \in W^{\o d}
\quad\tforall d\ge0
$$
define a unique mapping of invariants $F_\one\:\Inv(V,\ev_V,\coev_V)\to\Inv(W,\ev_W,\coev_W)$ in the sense of \Cref{def:mapping-invariants}.
Moreover, 
$$
F_\one(\tpx{11,12,13,14}{}{$k_1$,$k_2$,dots,$k_d$}) = \cO(E_{1,\theta^{-1}(k_1)}\o\dots\o E_{1,\theta^{-1}(k_d)})
\quad\text{for all } d\ge1,\quad k_1,\dots,k_d\in G,
$$
and $F_\one$ is surjective.
\end{lemma}

\begin{proof} The assignments define a unique graded linear map $F_\one$ of the desired form. $F_\one$ is an algebra map as the support of a tensor product $f_1\times f_2$ in the sense of \Cref{def:supp} is given by the product set of the supports of the tensor factors $f_i$. $F_\one$ is equivariant under the symmetric group actions, as the definition of support is compatible with these actions. 

We compute for all $d\ge2$, and any $G$-diagram $f\in P_{0,d}$ with labels $(k_1,\dots,k_d)$:
$$
(\ev_W\o W^{\o(d-2)}) F_\one(f) 
= \sum_{(i,j)\in\supp(f)} \delta_{i_1,i_2} \delta_{j_1,j_2} E_{i_3,j_3}\o\dots\o E_{i_d,j_d}
$$
and
$$
F_\one( (\ev_V\o V^{\o (d-2)}) f)
= \begin{cases}
    t F_\one(f') & \text{first two points in $f$ form a component and $k_1=k_2$} \\ 
    F_\one(f') & \text{first two points are in a bigger component and $k_1=k_2$} \\ 
    F_\one(f') & \text{first two points are not connected in $f$} \\ 
    0  & \text{else}
\end{cases}
$$
where $f'$ is the partition obtained by connecting the first two points in $f$, possibly after replacing the labels by equivalent labels such that $k_1=k_2$. The two expressions can be seen to be the same.

Moreover, 
$$
F_\one(\tpx{11,12,13,14}{}{$k_1$,$k_2$,dots,$k_d$})
 = \sum_{i_0\in\{1,\dots,n\},g\in G} E_{i_0,\theta^{-1}(gk_1)}\o\dots\o E_{i_0,\theta^{-1}(gk_d)} ,
$$
which is indeed $\cO(E_{1,\theta^{-1}(k_1)}\o\dots\o E_{1,\theta^{-1}(k_d)})$.

Hence, the induced map $\Inv_d(V)\xrightarrow{F_\one}\ol\Inv_d(W)$ is surjective, by \Cref{lem:generators-invariants}, which implies that $F_\one$ is surjective by \Cref{lem:reduction-surjective}.
\end{proof}

\begin{theorem}[{see~\cite{LS}*{Corollary~9.6}}] \label{thm:interpolation} If $t=n$, then there is an essentially surjective full LSM functor
$$F\:\Ps(\cG_t)\to\Rep G_n$$
extending $F_\one$ such that $F(V)=W$ and $F(\tp{1,2})=\ev_W$.
\end{theorem}

\begin{proof} 

\Cref{lem:helper} together with \Cref{prop:functors} implies that there is a full LSM functor $\cG_t \to \Rep G_n$ sending $\one\mapsto W$. This induces a full LSM functor $\Ps(\cG_t) \to \Rep G_n$. The latter functor is essentially surjective, as $W$ is a faithful $G_n$-module, so every simple $G_n$-module appears as a direct summand in a tensor power of $W_n$; hence, it is isomorphic to the image of some object under the above functor.
\end{proof}

We want to describe the kernel of the constructed functor. A $G$-diagram $f'$ is called a \emph{coarsening} of a $G$-diagram $f$ if the underlying partition diagram of $f'$ is a coarsening of that of $f$ and the labels can be chosen within the respective equivalence class of $f$ and $f'$ such that they agree. We write $f'<f$ whenever $f'$ is a proper coarsening of $f$ in this sense. For any morphism $h$, let $h^T$ denote the morphism obtained by swapping upper and lower points in all $G$-diagrams of $h$.

Recall that a morphism $f$ in an LSM category is \emph{negligible} if $gf$ has trace zero for any antiparallel morphism $g$.

\begin{definition} \label{def:x-of-f} For any $G$-diagram $f$, we define recursively
$$
x(f) = f - \sum_{f'<f} x(f') .
$$    
\end{definition}

\begin{lemma} \label{lem:helper-x-f} 
(a) Consider $a,b\ge0$, $f,g\in P_{a,b}$. Then $g^T x(f)=0$ if two lower points in $g$ are connected and the corresponding two lower points in $f$ are not connected. Similarly, $x(g)^T f=0$ if two lower points in $f$ are connected and the corresponding two lower points in $g$ are not connected.

(b)  For $a\ge0$, $f,g\in P_{0,a}$, we have $x(g)^T x(f)=\delta_{f,g}\, t(t-1)\dots(t-(\#f-1))$.

(c) For any $G$-diagram $f$ with more than $n$ components, $x(f)$ is negligible in $\cG_{t=n}$.
\end{lemma}

\begin{proof}
(a) Let $P'$ be the class of $G$-diagrams whose underlying partition diagram is an identity diagram. The special shape of the elements in $P'$ implies that pre- or post-composition with them just changes the labels of the upper or lower points, respectively. This implies that $x(hf)=hx(f)$ and $x(fh)=x(f)h$ for all $G$-diagram $f$ and all $h\in P'$, whenever the respective compositions make sense. 

On the other hand, we can use pre- and post-composition with elements from $P'$ to turn any $G$-diagram into a trivially labeled one: that is, for any $G$-diagram $f$, there are a $G$-diagram $f^0$ with the same underlying partition diagram and elements $h,h'\in P'$ such that $f=h f^0 h'$ and all labels are trivial for $f^0$.
For $f,g$ as in the first assertion, write $g=h_1 g^0 h_2$ and $f=h_3 f^0 h_4$ in this way, with $h_1,\dots,h_4\in P'$.

Let $C_{i,j}$ be the partition diagram obtained by adding a connection between the components of the $i$-th and the $j$-th upper point in the partition diagram $\id_b$. View can regard $C_{i,j}$ as a $G$-diagram with trivial labels. If $g$ has two distinct connected lower points $i$ and $j$, then the same applies to $g^0$, and we can write $g^0=C_{i,j} g^0$.

Now we compute
$$
g^T x(f) = h_2^T (g^0)^T C_{i,j} h_1^T x(h_3 f^0 h_4)
 = h_2^T (g^0)^T C_{i,j} h_1^T h_3 x(f^0) h_4 .
$$

The product $C_{i,j} h_1^T h_3$ can be zero, depending on the labels of $h_1^T h_3$, but in case it is not, it can be written as $h' C_{i,j}$ for some $h'\in P'$. This implies that the above morphism factors through $C_{i,j} x(f^0)$, where now $f^0$ is a $G$-diagram with trivial colors. This product is zero by \cite{Flake-Laugwitz}*{Claim in Proof of Lemma~3.13}. Hence, the first assertion follows. The second assertion is shown very similarly.

(b) $x(g)^T f$ is zero as soon as two points are connected in $g$ that are not connected in $f$ by (a). Similarly for $g^T x(f)$. Hence, the composition in question can only be non-zero if $f,g$ have the same underlying partition diagrams. In this case, if $f,g$ are not equivalent as $G$-partitions, then $(g')^T f'-9$ for all coarsening $f', g'$ of $f,g$, respectively. Hence, the composition in question can only be non-zero if $f,g$ are the same $G$-diagram. In this case, $(f')^T x(f)=0$ for any proper coarsening $f'$ of $f$. On the other hand,  $f^T f'=t^{\#f'}$ for any coarsening $f'$ of $f$. Hence, we compute
$$
f^T x(f) = t^{\# f} - \sum_{f'<f} f^T x(f') .
$$

Let $\mu(f,f')$ be the Möbius function of the lattices of set partitions, for partitions $f'\leq f$. It is given by 
$$
\mu(f,f') = (-1)^{\#f-\#f'} \prod_{C\in f'} (m_f(C)-1)! ,
$$
where $C$ ranges over the parts of $f'$ and $m_f(C)$ denotes the number of parts of $f$ contained in $C$.

Then by Möbius inversion,
$$
f^T x(f) 
= \sum_{f'\le f} t^{\#f'} \mu(f,f')
= \sum_{f'\le f} t^{\#f'} (-1)^{\#f -\#f'} \prod_{C\in f'} (m_f(C)-1)!
= \sum_{1\le k\le n} t^k (-t)^{n-k} c_{n,k} ,
$$
where $c_{n,k}$ is the number of permutations of $n$ with exactly $k$ disjoint cycles. Here, we use that any such permutation defines a coarsening of $f$ with exactly $k$ parts, where parts of $f$ are merged according to the disjoint cycle decomposition, and that any coarsening $f'$ of $f$ with exactly $k$ parts is obtained from exactly $\prod_{C\in f'} (m_f(C)-1)!$ permutations in this way.

In other words, $c_{n,k}$ is the unsigned Stirling number of the first kind, and the above sum is $t(t-1)\dots (t-n+1)$, as desired.

(c) This follows immediately from part (b), if $f$ has no upper points. For arbitrary diagrams this follows noting that they agree with diagram with no upper points up to a composition with suitable tensor products involving identities and (co-)evaluations.
\end{proof}

The following result up to the concrete description of the negligible morphisms was obtained in \cite{LS}*{Corollary~9.6} by identifying $\Ps(\cG_t)$ with a tensor envelope in the sense of \cite{Knop}:

\begin{theorem} \label{thm:interpolation-kernel}
If $t\not\in\ZZ_{\ge0}$, then $\Ps(\cG_t)$ is semisimple.

If $t\in\ZZ_{\ge0}$, then the kernel of $F\:\Ps(\cG_t)\to \Rep G_t$ is the tensor ideal of negligible morphisms in $\Ps(\cG_t)$, so $F$ is the semisimplification functor. Moreover, the negligible morphisms are spanned by morphisms of the form $x(f)$, for $f$ a $G$-partition with more than $t$ components. 
\end{theorem}

\begin{proof} By the theory of semisimplification \cite{BW} (and further developed in \cites{AK,Etingof-Ostrik}), we have to show that all negligible morphisms are zero for the first part, and are the kernel of $F$ for the second part.

Assume first $t\not\in\ZZ$, and consider the matrix $(\operatorname{tr}(x(g)^Tx(f))_{f,g}$, where $f,g$ range over all partitions of a fixed size $a,b\ge0$. By \Cref{lem:helper-x-f}(a,b), this matrix is a diagonal matrix with non-zero diagonal entries. Hence, it has trivial kernel, showing that all negligible morphisms must be zero.

Assume now $t\in\ZZ_{\ge0}$. As $\Rep G_t$ is semisimple, all negligible morphisms must be in the kernel of $F$. Now assume a morphism $h\neq0$ is in the kernel of $F$. Consider the matrix $(\operatorname{tr}(x(g)^Tx(f))_{f,g}$, where $f,g$ range over all partitions of the same size as $h$ with at most $t$ components. Then \Cref{lem:helper-x-f}(a,b) implies that this matrix is regular, so $h$ has to involve partitions with more than $t$ components. Pick a total order on the set of all $G$-diagram that refines the partial order given by $f'<f$ if $f'$ has less components than $f$. Pick the maximal $G$-diagram $f$ that appears in $h$ with non-zero coefficient $\alpha$. Then $\alpha x(f)$ is negligible by \Cref{lem:helper-x-f},  so $h':=h-\alpha f$ is in the kernel of $f$, but the maximal $G$-diagram appearing in $h'$ with non-zero coefficient is less than $f$, if it exists. Repeating the argument shows that $h$ has to be negligible.
\end{proof}

\section{Restriction and induction}

We continue to fix a group $G$.

\subsection{Diagonal functors} \label{sec:diagonal}

\newcommand\Decomp{\operatorname{Decomp}}

Fix $t,t_1,t_2\in\kk$ such that $t=t_1+t_2$.

Let $\cG_{t_1}\boxtimes\cG_{t_2}$ be the exterior tensor product category. It contains the object $V_\Delta:=([1]\boxtimes[0]) \oplus ([0]\boxtimes[1])$. The tensor powers of this object decompose as
$$
(V_\Delta)^{\o d}=([1]\boxtimes[0] \oplus [0]\boxtimes[1]) ^{\o d} \cong \bigoplus_{I\sqcup J=\{1,\dots,d\}} [\#I]\boxtimes[\#J] \quad 
\tforall t\ge0.
$$
Let $p_{I,J}$ and $q_{I,J}$ be the projection and embedding maps of this direct sum. They are defined for any disjoint subsets $I,J$ of $\ZZ_{\ge0}$ such that $I\cup J=\{1,\dots,\#I+\#J\}$.

For any $G$-diagram $f$, let $\Decomp(f)$ be the set of tuples $(I_1,I_2,J_1,J_2,f_1,f_2)$ such that $I_1,I_2,J_1,J_2$ are subsets of $\ZZ_{\ge0}$, and for $k=1,2$, $I_k,J_k$ are disjoint, $I_k\cup I_k=\{1,\dots,\#I_k+\#J_k\}$, and the upper and lower points of $f$ that lie in $I_k$ and $J_k$, respectively, form the $G$-diagram $f_k$. 

We write $f_1\sqcup f_2:=(I_1,I_2,J_1,J_2,f_1,f_2)$ for elements in $\Decomp(f)$, for short, omitting the subsets in the notation, and define
$$
[f_1 \boxtimes f_2] := q_{J_1,J_2} (f_1\boxtimes f_2) p_{I_1,I_2}
\quad \in(\cG_{t_1}\boxtimes\cG_{t_2})( (V_\Delta)^{\o m}, (V_\Delta)^{\o n} )
$$
for $f$ any $G$-diagram with $m$ upper and $n$ lower points, for any $f_1\sqcup f_2\in\Decomp(f)$.

Recall that $V\in\cG_t$ is a self-dual object with evaluation and coevaluation morphisms given by 
$$
\ev_V = \tp{1,2} , \qquad \coev_V=\tp{11,12}.
$$
It follows that $V_\Delta\in\cG_{t_1}\boxtimes\cG_{t_2}$ is a self-dual object with evaluation and coevaluation given by
$$
\ev_\Delta = q_{\emptyset,\emptyset} (\ev_V\boxtimes[0]) p_{\{1,2\},\emptyset} + q_{\emptyset,\emptyset} ([0]\boxtimes\ev_V) p_{\emptyset,\{1,2\}} ,
$$
$$
\coev_\Delta = q_{\{1,2\},\emptyset} (\coev_V\boxtimes[0]) p_{\emptyset,\emptyset}  +  p_{\emptyset,\{1,2\}} ([0]\boxtimes\coev_V) p_{\emptyset,\emptyset} . 
$$

We define
$$
D_\one\:\Inv_d(V,\ev_V,\coev_V)\to\Inv_d(V_\Delta,\ev_\Delta,\coev_\Delta),
\quad
f \mapsto \sum_{f_1\sqcup f_2\in\Decomp(f)} [f_1\boxtimes f_2] .
$$

\begin{lemma} $\label{lem:D-mapping-invariants} D_\one$ is a mapping of invariants in the sense of \Cref{def:mapping-invariants}.
\end{lemma}

\begin{proof} By definition, $D_\one$ is a graded map.

$D_\one$ is an algebra map, as for any $g_i\in\Inv_{d_i}(V)$, for $i=1,2$, we compute
\begin{flalign*}
D_\one(g_1\o g_2) 
&= \sum_{f_1\sqcup f_2\in\Decomp(g_1\o g_2)} [f_1\boxtimes f_2]
= \sum_{f_1\sqcup f_2\in\Decomp(g_1), f_3\sqcup f_4\in\Decomp(g_2)} [f_1\boxtimes f_2] [f_3\boxtimes f_4] .
\\
&= D_\one(g_1) D_\one(g_2).
\end{flalign*}

It follows directly that $D_\one$ is compatible with the $S_d$-actions for all $d\ge0$.

We check that 
$$
D_\one(\tp{11,12}) 
= q_{\{1,2\},\emptyset} (\tp{11,12}\boxtimes\emptyset) p_{\{1,2\},\emptyset}
+ q_{\emptyset,\{1,2\}} (\emptyset\boxtimes\tp{11,12}) p_{\emptyset,\{1,2\}}
=\ev_{V_\Delta}
. 
$$

For any $G$-diagram $f$ with at least two lower points, let $\gamma(f)$ be the $G$-diagram obtained by removing the first two lower points after merging the two blocks they are contained in.

Now let $f$ be any $G$-diagram with $d\ge2$ lower and no upper points. Let $p,q$ be the first two lower points in $f$. Then $(\ev_V\o V^{\o d-2}) f$ equals
$$
 \gamma(f) \begin{cases}
t  & p,q \text{ form a block and have the same color} \\
1 & p,q \text{ are not connected or have the same color, but do not form a block} \\
0 & \text{else}
\end{cases} .
$$
We also compute
\begin{flalign*}
(\ev_{V_\Delta}\o V_\Delta^{\o d-2}) D_\one(f) 
&= 
\sum_{f_1\sqcup f_2\in\Decomp(f), p,q\in f_1} [(\ev_V\o V^{\o\#\low(f_1)-2})f_1\boxtimes f_2]
\\&
+\sum_{f_1\sqcup f_2\in\Decomp(f), p,q\in f_2} [f_1\boxtimes (\ev_V\o V^{\o\#\low(f_2)-2})(f_2)] .
\end{flalign*}
Now if $p,q$ form a block and have the same color, then we can simplify this to
$$
t_1\sum_{f_1\sqcup f_2\in\Decomp(f), p,q\in f_1} [\gamma(f_1)\boxtimes f_2]
+t_2\sum_{f_1\sqcup f_2\in\Decomp(f), p,q\in f_2} [f_1\boxtimes\gamma(f_2)]
= t D_\one(\gamma(f)) .
$$
If $p,q$ are not connected or have the same color, but do not form a block, then we can simplify this to
$$
\sum_{f_1\sqcup f_2\in\Decomp(f), p,q\in f_1} [\gamma(f_1)\boxtimes f_2]
+\sum_{f_1\sqcup f_2\in\Decomp(f), p,q\in f_2} [f_1\boxtimes\gamma(f_2)]
= D_\one(\gamma(f)) .
$$
In the remaining cases, we indeed obtain $0$, which completes the proof.
\end{proof}

\begin{proposition} \label{prop:D} For any $t,t_1,t_2\in\kk$, $t=t_1+t_2$, there is an LSM functor
$$
D\: \cG_t\to \cG_{t_1}\boxtimes \cG_{t_2}  ,
\quad
f\mapsto \sum_{f_1\sqcup f_2\in\Decomp(f)} [f_1\boxtimes f_2] ,
$$
for any $G$-diagram $f$.
\end{proposition}

\begin{proof} The existence of the functor follows from \Cref{lem:D-mapping-invariants} by \Cref{prop:functors}. The formula follows from the explicit form of the functor derived in the proof of \Cref{prop:functors}.
\end{proof}

Note that the LSM functor $D$ induces then also an LSM functor
$$
D\: \Ps(\cG_t) \to \Ps(\cG_{t_1})\boxtimes\Ps(\cG_{t_2}) ,
$$
which we denote by the same symbol.

\subsection{Restriction and induction}
Fix $n\ge0$. We set $\cC_t:=\Ps(\cG_t)$, depending on our fixed choice of $G$, for any $t\in\kk$. Using the ``diagonal'' functor $D$ from \Cref{prop:D} and the ``specialization'' functor $F$ from \Cref{thm:interpolation}, we define: 
$$
\Res:=\Res_n:=(\Id\boxtimes F)D \: \cC_t\to\cC_{t-n}\boxtimes\Rep (G\wr S_n) 
\quad\text{for all }t\in\kk.
$$
with $D$ as in \Cref{sec:diagonal}. Our aim is to show that these ``restriction'' functors admit adjoint ``induction'' functors.

\bigskip

Assume from now on that $G$ is a finite group. Recall the definition of $x(\cdot)$ from \Cref{def:x-of-f}.

\begin{definition} We set
$$
x_n:=x(\id_{[n]}),
\quad
e_n:=(n!)^{-1}\sum_{\sigma\in S_n} \sigma,
\quad
\omega_n:=|G|^{-n}\sum_{k_1,\dots,k_n\in G} \tpx{1,11,0,3,13}{$1$,$\dots$,$1$}{$k_1$,$\dots$,$k_n$},
\quad
e:=e_n \omega_n x_n.
$$
\end{definition}

\begin{lemma} The elements $e_n,\omega_n,x_n$ are commuting idempotents in $\cC_t([n],[n])$, for all $t\in\kk$. With
$$
e' := \iota_{\emptyset,\{1,\dots,n\}} ([0]\boxtimes F(e)) \pi_{\emptyset,\{1,\dots,n\}} ,
$$
we have that $e'$ is an idempotent and
$$
e'\Res(e) = e' = \Res(e)e'  .
$$
\end{lemma}

\begin{proof} That $x_n$ is an idempotent follows from \Cref{lem:helper-x-f}(a). The element $e_n$ is an idempotent, as it is the image of an idempotent in the group algebra of $S_n$ under an algebra map. That $\omega_n$ is an idempotent follows from a short computation using that
$$
\tpx{1,11}{$1$}{$k$} \circ \tpx{1,11}{$1$}{$k'$} = \tpx{1,11}{$1$}{$k'k$}
\quad\tforall k,k'\in G.
$$

The element $x_n$ commutes with $e_n$ and $\omega_n$ as for any $G$-partition $f$ whose underlying partition diagram is a permutation diagram, we can see that $x_n f=x(f)$, as post-composition with all coarsenings of the identity $\id_{[n]}$ produces exactly all coarsening of $f$. That $e_n$ commutes with $\omega_n$ can be computed directly for all pairs of summands of the respective sums. 

It follows that $e$ is an idempotent, and from this that $e'$ is an idempotent, as well.

Let $f\in P_{n,n}$ be a $G$-partition without upper or lower components, i.e., one in which every component as upper and lower points. Then by the definition of $\Res$, 
$$
\pi_{\emptyset,\{1,\dots,n\}}\Res(f)
 = ([0]\o F(f)) \pi_{\emptyset,\{1,\dots,n\}} 
\quad\text{and}\quad
\Res(f)\iota_{\emptyset,\{1,\dots,n\}}
 = \iota_{\emptyset,\{1,\dots,n\}} ([0]\o F(f))  .
$$
This implies the asserted identity.
\end{proof}

\begin{definition} Set $X:=\im(e)$. Let $\pi'_X$ be the projection $\Res(X)=\im(\Res(e))\to\im(e')\cong [0]\boxtimes F(X)$.    
\end{definition}

\begin{theorem} \label{thm:res} $\Res_n$ has a simultaneous left and right adjoint functor sending $\one\mapsto X$.
\end{theorem}

\begin{proof} Set $\cD:=\cC_{t-n}\boxtimes\Rep (G\wr S_n)$. Define $p:=F(\pi'e)$, where $\pi'\in P_{n,0}$ is the unique $G$-partition with $n$ components.

By \cite{AbEnv1}*{Lemma~5.5}, we have to check that $\cD$ is idempotent complete, $\Res$ is dominant, and for all $a\ge0$, there map
$$
\phi: \cC_t([a], X) \to \cD(\Res([a]),\one),\quad
f\mapsto (\one\boxtimes p) \pi'_X \Res(f)
$$
is bijective. The first two conditions can be seen directly.

Let $F_n$ be the filtration on $\cD$ given by the number of components in $\cC_{t-n}$ in the first factor. We have
\begin{flalign*}    
\phi(f) 
&= (\one\boxtimes p)\pi'_X\Res(f)
 = \sum_{f_1\sqcup f_2\in\Decomp(f), \low(f_1)=\emptyset}
  (f_1 \boxtimes F(\pi' e f_2)) 
  \\
 &\in (f'_1 \boxtimes F(\pi' e f'_2))  
 + F_{\deg(f)-1} ,
\end{flalign*}
where $f'_1\sqcup f'_2\in\Decomp(f)$ is the decomposition in which $f'_1$ has no lower point, but maximal number of components otherwise.
Hence, the associated graded of $\phi$ sends
$$
\gr(\phi)(f) = f'_1 \boxtimes F(\pi' e f'_2)
$$
and it suffices to show that the latter map $\gr(\phi)$ is bijective. Hence, the assertion is reduced to \Cref{small-adjoint} below.
\end{proof}

\begin{lemma} \label{small-adjoint} For all $a\ge0$, the map
$$
\cC^+_t([a],X)\to \Rep (G\wr S_n)([a],\one), 
\quad 
f\mapsto F(\pi' e f) ,
$$
is an isomorphism, where $\cC^+_t([a],X)\subset\cC_t([a],X)$ is the subspace spanned by $ef$, for all $G$-partitions $f$ that have no upper components. 
\end{lemma}

\begin{proof} Consider a $G$-partition $f$ in $P_{a,n}$ that has no upper components. If $f$ has a component with two lower points, then $x_n f=0$ by \Cref{lem:helper-x-f}(a), so $ef=0$. So we may assume any component of $f$ has exactly one lower point. We may also assume the $i$-th lower point is connected to the $i$-th component whose left-most upper point is the $i$-th left-most upper point of any component, because for all $\sigma\in S_n$, $e_n \sigma=e_n$, so $e\sigma f=ef$. Finally, we may assume the $G$-labels of all lower points in $f$ and of all left-most upper points in each component are trivial, i.e., the identity element in $G$. Indeed, to achieve this for the left-most upper points in each component, a suitable representative can be chosen within the respective equivalence class, and then $\omega_n f$ does not depend on the labels of all lower points, so $ef$ does not, either. To sum up, we have established that the set $\{ef\}$, for $G$-partitions $f\in P_{a,n}$ such that
\begin{itemize}
\item all components of $f$ have exactly one lower point,
\item all lower points and all left-most upper points in each component have trivial $G$-label,
\item the lower points and the left-most upper points of all components have the same ordering,
\end{itemize}
are a spanning set of the hom-space on the left-hand side.

Consider now $\{F(\pi' e f)\}$, for $f$ as above. As $\pi' e_n=\pi'=\pi' \omega_n$ we have $\pi' e f=\pi' x_n f = x(\hat f)$, where $\hat f\in P_{a,0}$ is obtained from $f$ by removing all lower points. In other words, the set $\{F(\pi'ef)\}$ for $f$ as above is exactly the set 
$$
\{ F( x(\hat f) ):  \hat f\in P_{a,0} \text{ with $\le n$ components} 
\}
$$
which is a basis of the hom-space $\Rep(G\wr S_n)([a],\one)$ by \Cref{thm:interpolation-kernel}.

As each $G$-partition $f$ as above determines uniquely a $G$-partition $\hat f$ as above, and vice versa, we have shown the assertion.
\end{proof}

\begin{remark} \label{rem:comparison-Mori} The restriction functor we consider seems to agree with (a suitable special case of) the functor discussed in \cite{Mori}*{Section~4.4}. In \emph{loc.~cit.}, a so-called $*$-product is defined as an interpolation version of the induction functor. It seems to be implied that this $*$-product is indeed a functor, which then would be good candidate for the adjoint functor whose existence we prove in \Cref{thm:res}. However, no claims seem to be made regarding the $*$-product being the adjoint of the restriction functor. For $G=\{1\}$, that is, $\cC_t=\uRep(S_t)$ the relevant induction functors were considered in \cite{Repcomplexrank}*{Section~2.3}. In this case, their existence was also shown using ultraproduct methods in \cite{HK}*{Corollary/Definition~1.4.5}.
\end{remark}

\subsection{Abelian envelopes and connections to related theories} We assume that $G$ is a fixed finite group. As an application of our theory of restriction and induction functors, we show the existence of abelian envelopes for the interpolation categories $\cC_t=\Ps(\cG_t)$ for the groups $(G\wr S_n)_{n\ge0}$.

We recall the relevant definition from \cites{CEAH,CEOP}: for a pseudo-tensor category $\cC$ in the sense of \Cref{sec:lsm-functors}, an \emph{abelian envelope} is a faithful linear monoidal functor $\iota\:\cC\to\cT$, where $\cT$ is a tensor category in the sense of \Cref{sec:lsm-functors}, such that for every tensor category $\cT$, restriction along $\iota$ induces an equivalence between the category of exact linear monoidal functors of the form $\cT\to\cT'$ and the category of faithful linear monoidal functors of the form $\cC\to\cT'$. In the same situation, $\iota$ is called \emph{symmetric abelian envelope} if $\cC$, $\cT$, and $\iota$ are symmetric. In this case, $(\cT,\iota)$ satisfy a version of the above universal property with respect to symmetric functors (see \cite{CEOP}*{Lemma~2.5.1}).

In \cite{FG}, it was shown that some of the tensor envelopes in the sense of Knop (see \Cref{sec:Knop}) can be viewed as categories of tilting modules in a highest weight category whose underlying abelian category is the abelian envelope of the tensor envelope, assuming this abelian envelope exists. For the definitions related to highest weight categories we refer to \cite{BS}.

\begin{theorem} \label{thm:abenv} Assume that $\kk$ is algebraically closed of characteristic $0$. Then, for all $t\in\kk$, the categories $\cC_t=\Ps(\cG_t)$ have symmetric abelian envelopes with enough projectives. These abelian envelopes are lower finite highest weight categories in the sense of \cite{BS} such that $\cC_t$ is the subcategory of tilting objects 
\end{theorem}

\begin{proof} If $t\in\kk\setminus\ZZ_{\ge0}$, then the categories are semisimple tensor categories by \Cref{thm:interpolation-kernel}. Assume now $t\in\ZZ_{\ge0}$ and consider the restriction functor
$$
\cC_t \to \cC_{-1} \boxtimes \Rep (G\wr S_{t+1})
$$
as constructed in \Cref{sec:diagonal}. It is a linear monoidal functor with a right adjoint by \Cref{thm:res}. The category on the right-hand side is semisimple, as both factor categories are semisimple, and we assume $\kk$ to be algebraically closed, using \cite{FLP-cob}*{Theorem~5.11}. Now it follows from \cite{AbEnv1}*{Theorem~B} that $\cC_t$ has an abelian envelope with enough projectives. 

By \cite{AbEnv1}*{Theorem~C}, the abelian envelope is realized by category of finitely presented functors from the category of splitting objects in $\cC_t$ to abelian groups. In particular, this category and the embedding of $\cC_t$ are symmetric.

The abelian envelope is a lower finite highest weight category and $\cC_t$ can be identified with the subcategory of tilting objects by \cite{FG}*{Theorem~A}, using the equivalence of $\cC_t$ with a Knop tensor envelope shown in \cite{LS}*{Theorem~9.5} (see also \Cref{thm:equiv-2}).
\end{proof}

\begin{remark} \label{rem:HS} We describe an alternative strategy for establishing the existence of the abelian envelope via the theory of generalized permutation modules of pro-oligomorphic groups developed in \cite{HS}: by \cite{Sno}, there is a pro-oligomorphic group $G'$ and a measure $\mu$ on $G'$ in the sense of \cite{HS} associated to $\cC_t$, which can be viewed as a Knop tensor envelope by \cite{LS}*{Theorem~9.5} (or \Cref{thm:equiv-2}); see also \cite{FG}*{Remark~5.19}, where this general idea is explained.
Then by \cite{FG}*{Proposition~5.20}, it would suffice to show that $\mu$ is quasi-regular and has property (P) in the sense of \cite{HS} to see that $\cC_t$ has an abelian envelope. The abelian envelope would then be realized by a suitable category of modules over a certain completed group algebra in the sense of \cite{HS}.  
\end{remark}

\begin{remark} In \cite{Mori}, a class of ``wreath product'' interpolation categories $\cS_t(\cC)$ with parameter $t$, for a fixed suitable monoidal category $\cC$, is constructed that generalize the categories $\cC_t$, see \cite{Mori}*{Remark~4.25}. It would be interesting to see to which extent our theory of restriction and induction functors generalizes to these categories $\cS_t(\cC)$.
\end{remark}

\appendix
\section{Equivalences of categories} 
\label{sec:equivalences}

Fix a finite group $G$. In this appendix, we explain that the category constructed from $G$-partitions as in \Cref{sec:G-partitions} can equivalently be viewed as a category constructed from edge-labeled partition diagrams in the sense of Bloss \cite{Bloss}, or as a special case of a tensor envelope in the sense of Knop \cite{Knop}. These equivalences have been observed before, including in \cite{LS} and \cite{Mori}. We show how explicit functors realizing the equivalences can be obtained using the techniques from \Cref{sec:lsm-functors}.

\subsection{Vertex-labeled diagrams and edge-labeled diagrams}
We review the definition of a different type of $G$-labeled partition diagram, following \cite{Bloss}.

\begin{definition}
    \label{def:edge.-colored}
    An \emph{edge-labeled $G$-diagram} is an equivalence class of string diagrams where each edge is an arc between exactly two vertices, edges are labeled by elements from $G$, and each edge has an orientation (i.e.~direction). Two diagrams are considered equivalent if the underlying partition diagrams are equivalent and the edge labels differ locally according to following relations:
$$
   	\begin{tikzpicture}[scale=0.7]
   		\tikzstyle{every node}=[font=\small,inner sep=2pt]
   		\filldraw [black] (0,0) circle [radius=3pt];
   		\filldraw [black] (1,0) circle [radius=3pt];
   		
   		\filldraw [black] (3,0) circle [radius=3pt];
   		\filldraw [black] (4,0) circle [radius=3pt];
   		
   		\draw (2,0) node {$\sim$};
   		\begin{scope}[scale=1,
   			every path/.style={
   				postaction={nomorepostaction,decorate,
   					decoration={markings,mark=at position 0.5 with {\arrow{>}}}
   				}
   			}
   			]
   			
   			\draw (0,0)-- node[above]{$c$} (1,0);
   			\draw (4,0)-- node[above]{$c^{-1}$} (3,0);
   		\end{scope}
   	\end{tikzpicture}
   \qquad\text{and}\qquad
   	\begin{tikzpicture}[scale=0.7]
   		\tikzstyle{every node}=[font=\small,inner sep=2pt]
   		\filldraw [black] (0,0) circle [radius=3pt];
   		\filldraw [black] (1,0) circle [radius=3pt];
   		\filldraw [black] (0,-1) circle [radius=3pt];

   		\filldraw [black] (3,0) circle [radius=3pt];
   		\filldraw [black] (4,0) circle [radius=3pt];
   		\filldraw [black] (3,-1) circle [radius=3pt];
   		\draw (1.7,-0.5) node {$\sim$};
   		\begin{scope}[scale=1,
   			every path/.style={
   				postaction={nomorepostaction,decorate,
   					decoration={markings,mark=at position 0.5 with {\arrow{>}}}
   				}
   			}
   			]
   			\draw (0,0)-- node[above]{$c_2$} (1,0);
   			\draw (0,-1)-- node[left]{$c_1$} (0,0);

   			\draw (3,0)-- node[above]{$c_2$} (4,0);
   			\draw (3,-1)-- node[left]{$c_1$} (3,0);
   			\draw (3,-1)-- node[below,xshift=15pt]{$c_1 c_2$} (4,0);
   		\end{scope}
   	\end{tikzpicture} \qquad.
$$
   We also exclude all diagrams which are equivalent to a diagram containing edges as follows, for $c_1c_2\neq c_3$:\footnote{This condition in the definition seems to be implicit, but missing in \cite{Bloss}*{Section~6.2}. For instance, this can be seen from the asserted number of $(G,k)$-diagrams, in the terminology of the reference, see \cite{Bloss}*{bottom of page 704}.}
$$
   		\begin{tikzpicture}[scale=0.7]
   		\tikzstyle{every node}=[font=\small,inner sep=2pt]

   		\filldraw [black] (3,0) circle [radius=3pt];
   		\filldraw [black] (4,0) circle [radius=3pt];
   		\filldraw [black] (3,-1) circle [radius=3pt];
   		
   		\begin{scope}[scale=1,
   			every path/.style={
   				postaction={nomorepostaction,decorate,
   					decoration={markings,mark=at position 0.5 with {\arrow{>}}}
   				}
   			}
   			]

   			\draw (3,0)-- node[above]{$c_2$} (4,0);
   			\draw (3,-1)-- node[left]{$c_1$} (3,0);
   			\draw (3,-1)-- node[below,xshift=15pt]{$c_3$} (4,0);
   		\end{scope}

   	\end{tikzpicture}
    \qquad.
$$

\end{definition}
We can define two operations for edge-labeled $G$-diagrams as in \Cref{sec:interpolation}. The tensor product $f\o g$ for two such diagrams is the horizontal juxtaposition, and if the number $m$ of lower points of $f$ agrees with the number of upper points of $g$, then the $t$-composition $g\circ f$, for any $t\in\kk$, is defined as the equivalence class of diagrams obtained by gluing along the $m$ common points, multiplied by a scalar. Now, the scalar is $0$ if the edges coming from $f$ and those coming from $g$ are not compatible, and is $t^\ell$, where $\ell$ is the number of components that are not connected to upper and lower points after gluing, and that are removed. 

To describe the compatibility more precisely, consider any pair $i,j$ with $i<j$ of lower points of $f$, which are identified with upper points of $g$. If $i,j$ are not in the same component in either $f$ or $g$, there is no condition. If they are in the same component in $f$ and $g$, then we may assume there is an edge connecting $i$ and $j$ in both $f$ and $g$, as such an edge can be added with a unique edge label not changing the equivalence class of $f$ or $g$, respectively. We may also assume the edges have the same orientation, from $i$ to $j$, as we may invert the direction of any edge when also inverting the edge label. Now the condition is that the two edge labels coming from $f$ and $g$ agree for the pair $i,j$.

For instance, we get:
 $\begin{tikzpicture}[scale=0.7,baseline=0]
			\tikzstyle{every node}=[font=\small,inner sep=2pt]
			\filldraw [black] (0,0) circle [radius=3pt];
			\filldraw [black] (1,0) circle [radius=3pt];
			
			\begin{scope}[scale=1,
				every path/.style={
					postaction={nomorepostaction,decorate,
						decoration={markings,mark=at position 0.5 with {\arrow{>}}}
					}
				}
				]
				
				\draw (0,0) to[bend left=50] node[above]{$k$} (1,0);
				\draw (0,0) to[bend right=50] node[below]{$l$} (1,0);
			\end{scope}
		\end{tikzpicture} = \delta_{k,l} t$. 

In can be seen as in \Cref{sec:interpolation}, and is explained in \cite{Bloss}*{Section~6.2}, that these are well-defined operations. Let $E_{a,b}$ be the set of edge-labeled $G$-diagrams for all $a,b\ge0$. Let us interpret the $G$-diagrams $\id_a$ and $\Psi_{a,b}$ constructed in \Cref{sec:interpolation} also as edge-labeled $G$-diagrams with trivial edge-colors (i.e.~edge color $1\in G$ for all edges). It then follows that for any $t\in\kk$, there is a unique $\kk$-linear symmetric strict monoidal category $\cE_t$ whose objects are the set $\{[a]\}_{a\ge0}$ and whose morphism spaces are $\cE_t([a],[b])=\kk E_{a,b}$, with identities $\{\id_a\}_{a\ge0}$, with composition and tensor product as above, with tensor unit $[0]$, and with a symmetric braiding given by $\Psi_{a,b}$.  

\begin{proposition} \label{thm:equiv-1} $\cG_T$ and $\cE_t$ are isomorphic (in particular, equivalent) as LSM categories. 
\end{proposition}

\begin{proof}[Proof sketch] Recall from \Cref{sec:interpolation} that 
$$
V=[1], \ev_V = \tp{1,2},\quad \coev_V = \tp{11,12}
\qquad\text{in }\cG_t.
$$
We now set 
$$
U:=[1], \ev_U := \tp{1,2},\quad \coev_U := \tp{11,12}
\qquad\text{in }\cE_t,
$$
where the two diagrams have trivial edge labels (and hence, the choice of orientation does not matter). It follows immediately that $(U,\ev_U,\coev_U)$ defines a self-dual object in $\cE_t$.
By \Cref{prop:functors}, and as any object in $\cG$ and $\cE$ is a tensor power of $V$ and $U$, respectively, it suffices to construct a bijective mappings of invariants between the two categories with respect to the self-dual objects $(V,\ev_V,\coev_V)$ and $(U,\ev_U,\coev_U)$.

Define $\psi\:\Inv(V)\to\Inv(U)$ by sending a $G$-diagram $f$ with no upper points to the edge-labeled diagram $\psi(V)$ that has the same underlying partition diagram, with an arbitrary choice of orientation for all edges, and with edge labels given by $k_1^{-1} k_2$ for any edge whose start vertex has label $k_1\in G$ and whose end vertex has label $k_2\in G$ in $f$. It can be checked that this indeed yields a bijective mapping of invariants.
\end{proof}

\subsection{Vertex-labeled diagrams and Knop's tensor envelopes} \label{sec:Knop}
In \cite{Knop}, a general class of LSM categories are discussed whose construction is based on a calculus of relations for regular categories. We briefly review this construction.

Let $\cA$ be a subobject-finite regular exact Mal'cev category (in the sense of \cite{Knop}) whose terminal object $*$ has no proper subobject. Monomorphisms in $\cA$ are called \emph{injective} and extremal epimorphisms in $\cA$ are called \emph{surjective}.
Let $\delta$ be a \emph{degree function} for $\cA$ in the sense of \cite{Knop}, that is, essentially, a $\kk$-valued function on surjective morphisms in $\cA$ that is multiplicative with respect to compositions, invariant under pullbacks, and assigns the value $1$ to identities.

For any $A,B\in\cA$, let $\Rel_{A,B}$ be the set of \emph{relations}, i.e., subobjects of $A\times B$, where subobjects are injective maps up to an isomorphism of the source. An element in $\Rel_{A,B}$ can be viewed as a jointly injective span
$$
A\xleftarrow{} \cdot\xrightarrow{} B ,
$$
i.e., a span for which the induced map $\cdot\to A\times B$ is injective.
For any $A$, there is a distinguished element 
$$
\Delta_A = A \xleftarrow{\id} A \xrightarrow{\id} A \quad\in\Rel_{A,A} .
$$
For any pair of relations $f\in\Rel_{A,B}$, $g\in\Rel_{B,C}$, constructing a pullback along the maps to $B$, if possible, yields a span of the form $X\xleftarrow{} R\xrightarrow{} Y$, i.e., a morphism $h$ of the form $R\xrightarrow{}X\times Z$. An operation $\circ\:\kk\Rel_{Y,Z}\o_\kk \kk\Rel_{A,B}\to\kk\Rel_{X,Z}$ is defined by setting
$$
g\circ f :=\begin{cases} \delta(h_1) h_2 & \text{the above pullback exists} \\ 0 & \text{else} \end{cases} ,
$$
for all $f,g$ as above, where $(h_1,h_2)$ is the image factorization of $h$ in $\cA$, so $h_1$ is surjective and $h_2$ is injective.

Then $\cT^0:=\cT^0(\cA,\delta)$ is the $\kk$-linear symmetric strict monoidal category with
\begin{itemize}
\item objects: $[A]$ for isomorphism classes $A$ of objects in $\cA$;
\item hom-spaces: $\cT^0([A],[B]):=\kk\Rel_{A,B}$;
\item identities: given by $\Delta_A\in\Rel_{A,A}$,
\item composition: given by $\circ$ as defined above;
\item tensor product: given by the coproduct of relations, with tensor unit $[*]$;
\item symmetric braiding: given by spans of the form
$$
A\times B \xleftarrow{\id} A\times B \xrightarrow{(\pi_B,\pi_A)} B\times A .
$$
\end{itemize}

The pseudo-tensor category $\cT(\cA,\delta):=\Ps(\cT^0(\cA,\delta))$ is called \emph{Knop's tensor envelope} of $(\cA,\delta)$. In \cite{Knop}*{Section 9}, it is shown that many of these categories can be viewed as interpolation categories of families of finite groups, as well.

\begin{example} \label{expl:fin-set}
For instance, assume $\cA$ is the opposite category of the category of finite sets. Then, as is explained in \cite{Knop}, the degree functions $\delta$ are parametrized by elements $t\in\kk$, and for $t\in\kk$ with associated degree function $\delta_k$, the category $\cT^0(\cA,\delta_t)$ coincides with the category $\cS^0_t$, as in \Cref{bg:partitions}. 
More precisely, this equivalence is based on identifying jointly injective spans in $\cA$, the opposite category of the category of finite sets, with partition diagrams, as follows. Such a jointly injective span is the same as a jointly surjective cospan 
$$
A \xrightarrow{} C \xleftarrow{} B 
$$
of finite sets. We can view the sets $A,B$ as set of points and the set $C$ as a set of components. Then a jointly surjective cospan is the same as a partition of the set $A\sqcup B$, which is indeed the same as a partition diagram with upper points $A$, lower points $B$, and a set of components $C$.
\end{example}

Now let $\cA=(\GSet)\op$ be the opposite category of the category of finite free $G$-sets. Set $m:=|G|$. As explained in \cite{Knop}*{Section 8, Examples, 2.}, for any $t\in\kk$, this category has a degree function $\delta_t$ that assigns to an injective map $f$ of free $G$-sets of size $am$ and $bm$, respectively, the degree
$$
\delta_t(f) = t^{b-a} .
$$

Let $S$ be the free and transitive $G$-set with $m$ elements, i.e., the regular $G$-set. Then any finite free $G$-set is a coproduct $S^{\sqcup a}$ of $a$ copies of $S$, and is generated as a $G$-set by a choice of $a$ elements, one in each copy of $S$, for some $a\ge0$. Any $G$-map between free $G$-sets is given by the images such a fixed set of generators. Hence, any $G$-map $g\: S^{\sqcup a}\to S^{\sqcup b}$ is given by a map $set(g)\:\{1,\dots,a\}\to\{1,\dots,b\}$ indicating to which copy of $S$ in the target a given copy of $S$ in the source is mapped, and a function $rot(g)\:\{1,\dots,a\}\to G$ indicating the images of the generators in terms of the generator of the set $S$ to which they are mapped.

More formally, if $v_1,\dots,v_a$, $w_1,\dots,w_b$ are generators of $S^{\sqcup a}$ and $S^{\sqcup b}$, respectively, then 
$$
g(v_i) = rot(g)(i)\cdot w_{set(g)(i)}
\qquad\text{for all }i\in\{1,\dots,a\} .
$$

By definition, morphisms in $\cT^0(\cA,\delta_t)$ are linear combinations of cospans
$$
S^{\sqcup a} \xrightarrow{f_a} S^{\sqcup k} \xleftarrow{f_b} S^{\sqcup b}
\qquad\text{for }a,b,k\ge0
$$
that are jointly surjective. For any cospan as above, the cospan 
$$
\{1,\dots,a\}\xrightarrow{set(f_a)}\{1,\dots,k\}\xleftarrow{set(f_b)}\{1,\dots,b\}
$$ in the category of finite sets is jointly surjective, too, and hence, defines a partition diagram $\pi_{f_a,f_b}$ with $a$ upper and $b$ lower points (see \Cref{expl:fin-set}). Let $c_{f_a,f_b}$ be the coloring of this partition diagram given by the maps $rot(f_a)$ and $rot(f_b)$.

It can be verified that this construction yields a bijection between the relations in $\cA$ between the objects $S^{\sqcup a}$ and $S^{\sqcup b}$ and the set of $G$-diagrams with $a$ upper and $b$ lower points, i.e.,
\begin{equation} \label{eq:bijection}
\Rel_{S^{\sqcup a},S^{\sqcup b}} \leftrightarrow P_{a,b}
\quad\tforall a,b\ge0 .
\end{equation}

\begin{proposition}[\cite{LS}*{Theorem~9.5}]  \label{thm:equiv-2}  $\cG_t$ and $\cT^0(\cA,\delta_t)$ are isomorphic (in particular, equivalent) as LSM categories.
\end{proposition}

\begin{proof}[Proof sketch.]
Note that in the LSM category $\cT^0(\cA,\delta_t)$, the object $[S]$ is self-dual with evaluation and coevaluation given by
$$
\ev_S\: S^{\sqcup 2}\xrightarrow{(\id,\id)} S\xleftarrow{} \emptyset ,\quad
\coev_S\: \emptyset\xrightarrow{} S\xleftarrow{(\id,\id)} S^{\sqcup 2} .
$$
Indeed, this is explained at \cite{Knop}*{Equation~(3.15)}, and follows from a direct computation using that the pushout of the identity morphisms is given by identity morphisms.

As in the proof of \Cref{thm:equiv-1}, one can prove the assertion by constructing a bijective mapping of invariants between the two categories with respect to the self-dual objects $(S,\ev_S,\coev_S)$ and $(V,\ev_V,\coev_V)$. Indeed, one can verify that such a mapping $\Inv(S)\to\Inv(V)$ is obtained by extending the bijection from \Cref{eq:bijection} linearly. 
\end{proof}

\bibliography{bib}{}

@article {Knop,
    AUTHOR = {Knop, Friedrich},
     TITLE = {Tensor envelopes of regular categories},
   JOURNAL = {Adv. Math.},
  FJOURNAL = {Advances in Mathematics},
    VOLUME = {214},
      YEAR = {2007},
    NUMBER = {2},
     PAGES = {571--617},
      ISSN = {0001-8708,1090-2082},
   MRCLASS = {18E30 (20C15 20G05)},
  MRNUMBER = {2349713},
       DOI = {10.1016/j.aim.2007.03.001},
       URL = {https://doi.org/10.1016/j.aim.2007.03.001},
}

@article {Bloss,
    AUTHOR = {Bloss, Matthew},
     TITLE = {{$G$}-colored partition algebras as centralizer algebras of
              wreath products},
   JOURNAL = {J. Algebra},
  FJOURNAL = {Journal of Algebra},
    VOLUME = {265},
      YEAR = {2003},
    NUMBER = {2},
     PAGES = {690--710},
      ISSN = {0021-8693,1090-266X},
   MRCLASS = {20C99 (05E10)},
  MRNUMBER = {1987025},
MRREVIEWER = {Michael\ A.\ Zabrocki},
       DOI = {10.1016/S0021-8693(03)00132-7},
       URL = {https://doi.org/10.1016/S0021-8693(03)00132-7},
}

@article {CO-blocks,
    AUTHOR = {Comes, Jonathan and Ostrik, Victor},
     TITLE = {On blocks of {D}eligne's category {$\underline{\rm Re}{\rm
              p}(S_t)$}},
   JOURNAL = {Adv. Math.},
  FJOURNAL = {Advances in Mathematics},
    VOLUME = {226},
      YEAR = {2011},
    NUMBER = {2},
     PAGES = {1331--1377},
      ISSN = {0001-8708,1090-2082},
   MRCLASS = {20C30 (18E05)},
  MRNUMBER = {2737787},
MRREVIEWER = {Karin\ Erdmann},
       DOI = {10.1016/j.aim.2010.08.010},
       URL = {https://doi.org/10.1016/j.aim.2010.08.010},
}

@incollection {Deligne,
    AUTHOR = {Deligne, P.},
     TITLE = {La cat\'egorie des repr\'esentations du groupe sym\'etrique
              {$S_t$}, lorsque {$t$} n'est pas un entier naturel},
 BOOKTITLE = {Algebraic groups and homogeneous spaces},
    SERIES = {Tata Inst. Fund. Res. Stud. Math.},
    VOLUME = {19},
     PAGES = {209--273},
 PUBLISHER = {Tata Inst. Fund. Res., Mumbai},
      YEAR = {2007},
      ISBN = {978-81-7319-802-1},
   MRCLASS = {20C30 (18D10)},
  MRNUMBER = {2348906},
MRREVIEWER = {Karin\ Erdmann},
}

@mastersthesis{Thesis,
    author = {David Hull},
    title = {Restriction Functors on
Interpolation Categories},
    school = {University of Bonn},
    year = {2025}
}

@article {Repcomplexrank,
    AUTHOR = {Etingof, Pavel},
     TITLE = {Representation theory in complex rank, {I}},
   JOURNAL = {Transform. Groups},
  FJOURNAL = {Transformation Groups},
    VOLUME = {19},
      YEAR = {2014},
    NUMBER = {2},
     PAGES = {359--381},
      ISSN = {1083-4362,1531-586X},
   MRCLASS = {20C30 (20C05)},
  MRNUMBER = {3200430},
MRREVIEWER = {Bhama\ Srinivasan},
       DOI = {10.1007/s00031-014-9260-2},
       URL = {https://doi.org/10.1007/s00031-014-9260-2},
}

@article {FLP-cob,
    AUTHOR = {Flake, Johannes and Laugwitz, Robert and Posur, Sebastian},
     TITLE = {Indecomposable objects in {K}hovanov-{S}azdanovic's
              generalizations of {D}eligne's interpolation categories},
   JOURNAL = {Adv. Math.},
  FJOURNAL = {Advances in Mathematics},
    VOLUME = {415},
      YEAR = {2023},
     PAGES = {Paper No. 108892, 70},
      ISSN = {0001-8708,1090-2082},
   MRCLASS = {18M05 (05A18 17B10 18M30 57R56 81R05)},
  MRNUMBER = {4544563},
MRREVIEWER = {Benjamin\ Cooper},
       DOI = {10.1016/j.aim.2023.108892},
       URL = {https://doi.org/10.1016/j.aim.2023.108892},
}

@article{AbEnv1,
      title={Monoidal adjunctions and abelian envelopes}, 
      author={Johannes Flake and Robert Laugwitz and Sebastian Posur},
      year={2026},
      eprint={arXiv:2601.16092},
      primaryClass={math.RT},
      url={https://arxiv.org/abs/2601.16092}, 
}

@article {Flake-Laugwitz,
    AUTHOR = {Flake, Johannes and Laugwitz, Robert},
     TITLE = {On the monoidal center of {D}eligne's category {$\underline
              {\rm Re}{\rm p}(S_t)$}},
   JOURNAL = {J. Lond. Math. Soc. (2)},
  FJOURNAL = {Journal of the London Mathematical Society. Second Series},
    VOLUME = {103},
      YEAR = {2021},
    NUMBER = {3},
     PAGES = {1153--1185},
      ISSN = {0024-6107,1469-7750},
   MRCLASS = {18M15 (05E16 57K14)},
  MRNUMBER = {4245833},
       DOI = {10.1112/jlms.12403},
       URL = {https://doi.org/10.1112/jlms.12403},
}

@incollection {Etingof-Ostrik,
    AUTHOR = {Etingof, Pavel and Ostrik, Victor},
     TITLE = {On semisimplification of tensor categories},
 BOOKTITLE = {Representation theory and algebraic geometry---a conference
              celebrating the birthdays of {S}asha {B}eilinson and {V}ictor
              {G}inzburg},
    SERIES = {Trends Math.},
     PAGES = {3--35},
 PUBLISHER = {Birkh\"auser/Springer, Cham},
      YEAR = {2022},
      ISBN = {978-3-030-82006-0; 978-3-030-82007-7},
   MRCLASS = {18M20 (18M25 20D20 20J15)},
  MRNUMBER = {4486913},
MRREVIEWER = {Johannes\ Flake},
       DOI = {10.1007/978-3-030-82007-7\_1},
       URL = {https://doi.org/10.1007/978-3-030-82007-7\_1},
}

@article {BS,
    AUTHOR = {Brundan, Jonathan and Stroppel, Catharina},
     TITLE = {Semi-infinite highest weight categories},
   JOURNAL = {Mem. Amer. Math. Soc.},
  FJOURNAL = {Memoirs of the American Mathematical Society},
    VOLUME = {293},
      YEAR = {2024},
    NUMBER = {1459},
     PAGES = {vii+152},
      ISSN = {0065-9266,1947-6221},
      ISBN = {978-1-4704-6783-8; 978-1-4704-7716-5},
   MRCLASS = {16-02 (08C20 16G10 17B10 17B67 18E10 20G05 20G42)},
  MRNUMBER = {4684337},
       DOI = {10.1090/memo/1459},
       URL = {https://doi.org/10.1090/memo/1459},
}

@article {CEOP,
    AUTHOR = {Coulembier, Kevin and Etingof, Pavel and Ostrik, Victor and
              Pauwels, Bregje},
     TITLE = {Monoidal abelian envelopes with a quotient property},
   JOURNAL = {J. Reine Angew. Math.},
  FJOURNAL = {Journal f\"ur die Reine und Angewandte Mathematik. [Crelle's
              Journal]},
    VOLUME = {794},
      YEAR = {2023},
     PAGES = {179--214},
      ISSN = {0075-4102,1435-5345},
   MRCLASS = {18M05 (14L17)},
  MRNUMBER = {4529414},
MRREVIEWER = {Johannes\ Flake},
       DOI = {10.1515/crelle-2022-0076},
       URL = {https://doi.org/10.1515/crelle-2022-0076},
}

@article{FG,
      title={Monoidal Ringel duality and monoidal highest weight envelopes}, 
      author={Johannes Flake and Jonathan Gruber},
      year={2026},
      eprint={arXiv:2512.19558},
      archivePrefix={arXiv},
      primaryClass={math.RT},
      url={https://arxiv.org/abs/2512.19558}, 
}

@article{HS,
      title={Oligomorphic groups and tensor categories}, 
      author={Nate Harman and Andrew Snowden},
      year={2024},
      eprint={arXiv:2204.04526},
      archivePrefix={arXiv},
      primaryClass={math.RT},
      url={https://arxiv.org/abs/2204.04526}, 
}

@article{Sno,
      title={Regular categories, oligomorphic monoids, and tensor categories}, 
      author={Andrew Snowden},
      year={2024},
      eprint={arXiv:2403.16267},
      archivePrefix={arXiv},
      primaryClass={math.RT},
      url={https://arxiv.org/abs/2403.16267}, 
}

@article {Mori,
    AUTHOR = {Mori, Masaki},
     TITLE = {On representation categories of wreath products in
              non-integral rank},
   JOURNAL = {Adv. Math.},
  FJOURNAL = {Advances in Mathematics},
    VOLUME = {231},
      YEAR = {2012},
    NUMBER = {1},
     PAGES = {1--42},
      ISSN = {0001-8708,1090-2082},
   MRCLASS = {18D10},
  MRNUMBER = {2935382},
MRREVIEWER = {Eric\ C.\ Rowell},
       DOI = {10.1016/j.aim.2012.05.002},
       URL = {https://doi.org/10.1016/j.aim.2012.05.002},
}

@article{AbEnv2,
      title={Restriction, induction, and abelian envelopes for interpolation and cobordism categories}, 
      author={Johannes Flake and Robert Laugwitz and Sebastian Posur},
      year={2026?},
      note={In preparation}
}

@article {LS,
    AUTHOR = {Nyobe Likeng, Samuel and Savage, Alistair},
     TITLE = {Group partition categories},
   JOURNAL = {J. Comb. Algebra},
  FJOURNAL = {Journal of Combinatorial Algebra},
    VOLUME = {5},
      YEAR = {2021},
    NUMBER = {4},
     PAGES = {369--406},
      ISSN = {2415-6302,2415-6310},
   MRCLASS = {18M30 (20E22 20J15)},
  MRNUMBER = {4344790},
MRREVIEWER = {M.\ M.\ Al-Shomrani},
       DOI = {10.4171/jca/55},
       URL = {https://doi.org/10.4171/jca/55},
}

@article {Harman1,
    AUTHOR = {Harman, Nate},
     TITLE = {Generators for the representation rings of certain wreath
              products},
   JOURNAL = {J. Algebra},
  FJOURNAL = {Journal of Algebra},
    VOLUME = {445},
      YEAR = {2016},
     PAGES = {125--135},
      ISSN = {0021-8693,1090-266X},
   MRCLASS = {20C15 (20C30)},
  MRNUMBER = {3418050},
MRREVIEWER = {Lucio\ Centrone},
       DOI = {10.1016/j.jalgebra.2015.09.003},
       URL = {https://doi.org/10.1016/j.jalgebra.2015.09.003},
}

@article {Ryba1,
    AUTHOR = {Ryba, Christopher},
     TITLE = {Stable {G}rothendieck rings of wreath product categories},
   JOURNAL = {J. Algebraic Combin.},
  FJOURNAL = {Journal of Algebraic Combinatorics. An International Journal},
    VOLUME = {49},
      YEAR = {2019},
    NUMBER = {3},
     PAGES = {267--307},
      ISSN = {0925-9899,1572-9192},
   MRCLASS = {19A99 (16E20 18M05)},
  MRNUMBER = {3945269},
MRREVIEWER = {Andrea\ Solotar},
       DOI = {10.1007/s10801-018-0856-9},
       URL = {https://doi.org/10.1007/s10801-018-0856-9},
}

@article {Ryba2,
    AUTHOR = {Ryba, Christopher},
     TITLE = {The structure of the {G}rothendieck rings of wreath product
              {D}eligne categories and their generalisations},
   JOURNAL = {Int. Math. Res. Not. IMRN},
  FJOURNAL = {International Mathematics Research Notices. IMRN},
      YEAR = {2021},
    NUMBER = {16},
     PAGES = {12420--12462},
      ISSN = {1073-7928,1687-0247},
   MRCLASS = {18M15 (19A22)},
  MRNUMBER = {4300230},
MRREVIEWER = {Markus\ Szymik},
       DOI = {10.1093/imrn/rnz144},
       URL = {https://doi.org/10.1093/imrn/rnz144},
}

@article {Freslon-Skalski,
    AUTHOR = {Freslon, Amaury and Skalski, Adam},
     TITLE = {Wreath products of finite groups by quantum groups},
   JOURNAL = {J. Noncommut. Geom.},
  FJOURNAL = {Journal of Noncommutative Geometry},
    VOLUME = {12},
      YEAR = {2018},
    NUMBER = {1},
     PAGES = {29--68},
      ISSN = {1661-6952,1661-6960},
   MRCLASS = {20E22 (05A18 05E10 20G42)},
  MRNUMBER = {3782053},
MRREVIEWER = {Shilin\ Yang},
       DOI = {10.4171/JNCG/270},
       URL = {https://doi.org/10.4171/JNCG/270},
}

@article {CO-ab,
    AUTHOR = {Comes, Jonathan and Ostrik, Victor},
     TITLE = {On {D}eligne's category {$\underline{\rm Re}{\rm
              p}^{ab}(S_d)$}},
   JOURNAL = {Algebra Number Theory},
  FJOURNAL = {Algebra \& Number Theory},
    VOLUME = {8},
      YEAR = {2014},
    NUMBER = {2},
     PAGES = {473--496},
      ISSN = {1937-0652,1944-7833},
   MRCLASS = {18D10},
  MRNUMBER = {3212864},
MRREVIEWER = {Alessandro\ Ardizzoni},
       DOI = {10.2140/ant.2014.8.473},
       URL = {https://doi.org/10.2140/ant.2014.8.473},
}

@article{CEAH,
 author = {Coulembier, Kevin and Entova-Aizenbud, Inna and Heidersdorf, Thorsten},
 title = {Monoidal abelian envelopes and a conjecture of {Benson} and {Etingof}},
 fjournal = {Algebra \& Number Theory},
 journal = {Algebra Number Theory},
 issn = {1937-0652},
 volume = {16},
 number = {9},
 pages = {2099--2117},
 year = {2022},
 language = {English},
 doi = {10.2140/ant.2022.16.2099},
 zbMATH = {7640151},
 Zbl = {1503.18006}
}

@article{Cou-ab,
 author = {Coulembier, Kevin},
 title = {Monoidal abelian envelopes},
 fjournal = {Compositio Mathematica},
 journal = {Compos. Math.},
 issn = {0010-437X},
 volume = {157},
 number = {7},
 pages = {1584--1609},
 year = {2021},
 language = {English},
 doi = {10.1112/S0010437X21007399},
 zbMATH = {7374902},
 Zbl = {1471.18020}
}

@article{HT,
 author = {Heidersdorf, Th. and Tyriard, G.},
 title = {On interpolation categories for the hyperoctahedral group},
 fjournal = {Algebras and Representation Theory},
 journal = {Algebr. Represent. Theory},
 issn = {1386-923X},
 volume = {28},
 number = {2},
 pages = {613--646},
 year = {2025},
 language = {English},
 doi = {10.1007/s10468-025-10331-y},
 zbMATH = {8053883},
 Zbl = {1570.17014}
}

@article{BW,
 author = {Barrett, John W. and Westbury, Bruce W.},
 title = {Spherical categories},
 fjournal = {Advances in Mathematics},
 journal = {Adv. Math.},
 issn = {0001-8708},
 volume = {143},
 number = {2},
 pages = {357--375},
 year = {1999},
 language = {English},
 doi = {10.1006/aima.1998.1800},
 zbMATH = {1295253},
 Zbl = {0930.18004}
}

@article{AK,
 author = {Andr{\'e}, Yves and Kahn, Bruno},
 title = {Nilpotence, radicals and monoidal structures. {With} an appendix by {Peter} {O}'{Sullivan}.},
 fjournal = {Rendiconti del Seminario Matematico della Universit{\`a} di Padova},
 journal = {Rend. Semin. Mat. Univ. Padova},
 issn = {0041-8994},
 volume = {108},
 pages = {107--291},
 year = {2002},
 language = {French},
 url = {https://eudml.org/doc/108589},
 zbMATH = {5019695},
 Zbl = {1165.18300}
}

@article {HK,
    AUTHOR = {Harman, Nate and Kalinov, Daniil},
     TITLE = {Classification of simple algebras in the {D}eligne category
              {${\rm Rep}(S_t)$}},
   JOURNAL = {J. Algebra},
  FJOURNAL = {Journal of Algebra},
    VOLUME = {549},
      YEAR = {2020},
     PAGES = {215--248},
      ISSN = {0021-8693,1090-266X},
   MRCLASS = {17B10 (03C20 16B50 18B99)},
  MRNUMBER = {4050674},
MRREVIEWER = {Xiao-Wu\ Chen},
       DOI = {10.1016/j.jalgebra.2019.12.010},
       URL = {https://doi.org/10.1016/j.jalgebra.2019.12.010},
}
\bibliographystyle{plain}
\end{document}